\documentclass[11pt,reqno,twoside]{amsart}

\usepackage{amsfonts,amsmath,amssymb}
\usepackage{mathrsfs,mathtools}
\usepackage{enumerate}
\usepackage{hyperref}
\usepackage{esint}
\usepackage{graphicx}
\usepackage{bm}
\usepackage{commath}
\usepackage{esint}
\DeclareMathAlphabet{\mathpzc}{OT1}{pzc}{m}{it}

\usepackage{placeins} 
\usepackage{flafter} 

\usepackage{caption}

\numberwithin{equation}{section}

\usepackage[dvips,bottom=1.4in,right=1in,top=1in, left=1in]{geometry}

\makeatletter
\def\eqnarray{\stepcounter{equation}\let\@currentlabel=\theequation
\global\@eqnswtrue
\tabskip\@centering\let\\=\@eqncr
$$\halign to \displaywidth\bgroup\hfil\global\@eqcnt\z@
  $\displaystyle\tabskip\z@{##}$&\global\@eqcnt\@ne
  \hfil$\displaystyle{{}##{}}$\hfil
  &\global\@eqcnt\tw@ $\displaystyle{##}$\hfil
  \tabskip\@centering&\llap{##}\tabskip\z@\cr}

\def\endeqnarray{\@@eqncr\egroup
      \global\advance\c@equation\m@ne$$\global\@ignoretrue}

\newtheorem{theorem}{Theorem}[section]

\newtheorem{definition}[theorem]{Definition}

\newtheorem{lemma}[theorem]{Lemma}

\newtheorem{proposition}[theorem]{Proposition}
\newtheorem{assumption}[theorem]{Assumption}
\newtheorem{remark}[theorem]{Remark}
\numberwithin{equation}{section}

\def\Omc{\mathbb{R}^N\setminus\Omega}

\def\RR{{\mathbb{R}}}
\def\NN{{\mathbb{N}}}

\def\Om{\Omega}

\def\bOm{\overline{\Om}}
\def\pOm{\partial\Omega}

\title{External optimal control of fractional parabolic PDEs}

\author{Harbir Antil}
\address{Department of Mathematical Sciences, George Mason University, Fairfax, VA 22030, USA.}
\email{hantil@gmu.edu}

\author{Deepanshu Verma}
\address{Department of Mathematical Sciences, George Mason University, Fairfax, VA 22030, USA.}
\email{dverma2@gmu.edu}

\author{Mahamadi Warma}
\address{University of Puerto Rico  (Rio Piedras Campus), College of Natural Sciences,
Department of Mathematics, PO Box 70377 San Juan PR
00936-8377 (USA). }
\email{mahamadi.warma1@upr.edu, mjwarma@gmail.com}

\thanks{The first and second authors are partially supported by NSF grants DMS-1521590, DMS-1818772 and the Air Force Office of Scientific Research under Award NO: FA9550-19-1-0036. 
The third author is partially supported by the Air Force Office of Scientific Research 
under Award NO:  FA9550-18-1-0242}



\begin{document}

\begin{abstract}
In this paper we introduce a new notion of optimal control, or source identification in inverse, problems with fractional parabolic PDEs as constraints. This new notion allows a  source/control placement outside the domain where the PDE is fulfilled. We tackle the Dirichlet, the Neumann and the Robin cases. For the fractional elliptic PDEs this has been recently investigated by the authors in \cite{HAntil_RKhatri_MWarma_2018a}. 
The need for these novel optimal control concepts stems  from the fact that the classical PDE models only allow placing the source/control either on the boundary or in the interior where the PDE is satisfied. However, the nonlocal behavior of the fractional operator now allows placing the control in the exterior. We introduce the notions of weak and very-weak solutions to the parabolic Dirichlet problem. We present an approach on how to approximate the parabolic  Dirichlet solutions by the parabolic Robin solutions (with convergence rates). A complete analysis for the Dirichlet and Robin optimal control problems has been discussed. The numerical examples confirm our theoretical findings and further illustrate the potential 
benefits of nonlocal models over the local ones. 
\end{abstract}

\maketitle

\section{Introduction}

Let $\Om \subset \RR^N$, $N \ge 1$, be a bounded open set with boundary $\pOm$. Consider the Banach spaces $(Z_D,U_D)$ and $(Z_{R},U_R)$, where the subscripts $D$ and $R$ denote Dirichlet and Robin. The goal of this paper is to study the following parabolic external optimal control 
(or source identification) problems: 
\begin{itemize}
 \item \emph{\bf Fractional parabolic Dirichlet exterior control (source identification) problem:} Given $\xi \ge 0$
 a constant penalty parameter we consider the minimization problem:
 \begin{subequations}\label{eq:dcp}
 \begin{equation}\label{eq:Jd}
    \min_{(u,z)\in (U_D,Z_D)} J(u) + \frac{\xi}{2} \|z\|^2_{Z_D} ,
 \end{equation}
 subject to the fractional parabolic Dirichlet exterior value problem: Find $u \in U_D$ solving 
 \begin{equation}\label{eq:Sd}
 \begin{cases}
    \partial_t u + (-\Delta)^s u = 0 \quad &\mbox{in } Q:=(0,T) \times \Om, \\
                u = z           \quad &\mbox{in } \Sigma := (0,T) \times (\Omc),  \\
                u(0,\cdot)  = 0 \quad &\mbox{in } \Om , 
 \end{cases}                
 \end{equation} 
 and the control constraints 
 \begin{equation}\label{eq:Zd}
    z \in Z_{ad,D} ,
 \end{equation}
 \end{subequations}
 with $Z_{ad,D} \subset Z_D$ being a closed and convex subset. \\

\item \emph{\bf Fractional parabolic Robin exterior control (source identification) problem:} Given $\xi \ge 0$
 a constant penalty parameter we consider the minimization problem
 \begin{subequations}\label{eq:ncp}
 \begin{equation}\label{eq:Jn}
    \min_{(u,z)\in (U_R, Z_R)} J(u) + \frac{\xi}{2} \|z\|^2_{Z_R} ,
 \end{equation}
 subject to the fractional parabolic Robin exterior value problem: Find $u \in U_R$ solving 
 \begin{equation}\label{eq:Sn}
 \begin{cases}
    \partial_t u + (-\Delta)^s u = 0 \quad &\mbox{in } Q, \\
    \mathcal{N}_s u  +\kappa u = \kappa z \quad &\mbox{in } \Sigma, \\
    u(0,\cdot) = 0 \quad &\mbox{in } \Om ,
 \end{cases}                
 \end{equation} 
 and the control constraints 
 \begin{equation}\label{eq:Zn}
    z \in Z_{ad,R} ,
 \end{equation}
 \end{subequations}
 with $Z_{ad,R} \subset Z_R$ being a closed and convex subset. In \eqref{eq:Sn}, $\mathcal N_su$ denotes the interaction operator and 
is given in \eqref{NLND} below, $\kappa\in L^1(\Omc)\cap L^\infty(\Omc)$ and is non-negative. We notice that the latter assumption is not a restriction since otherwise we can replace $\kappa$ throughout by $|\kappa|$.
 \end{itemize}
Notice that \eqref{eq:Sn} is a generalized exterior value problem and all the details (with minor modifications) transfer to the case when instead of 
$\mathcal{N}_s u  +\kappa u = \kappa z \quad \mbox{in } \Sigma$ we consider 
$\mathcal{N}_s u  =  \kappa z \quad \mbox{in } \Sigma$, where $z$ denotes the control/source. The resulting optimal control problem is the parabolic Neumann exterior control problem. We mention that,  we can also deal with the following more general system:
\begin{equation*}
 \begin{cases}
    \partial_t u + (-\Delta)^s u = f \quad &\mbox{in } Q, \\
    \mathcal{N}_s u  +\kappa u = \kappa z \quad &\mbox{in } \Sigma, \\
    u(0,\cdot) = u_0 \quad &\mbox{in } \Om . 
 \end{cases}                
 \end{equation*} 
 In fact, one has to decompose the solution $u$ of the above system as $u=u_1+u_2$, where $u_1$ satisfies  \eqref{eq:Sn} and $u_2$ solves the system
\begin{equation*}
 \begin{cases}
    \partial_t u_2 + (-\Delta)^s u_2 = f \quad &\mbox{in } Q, \\
    \mathcal{N}_s u_2  +\kappa u_2 =0 \quad &\mbox{in } \Sigma, \\
    u_2(0,\cdot) = u_0 \quad &\mbox{in } \Om ,
 \end{cases}                
 \end{equation*} 
 and use some semigroups method (since in that case $u_2$ is given by  a semigroup).

The classical parabolic models, such as diffusion equation, are too restrictive. They 
only allow a source/control placement either inside the domain $\Om$ or on the boundary of 
the domain $\partial\Om$. Notice that in both \eqref{eq:Sd} and \eqref{eq:Sn} the 
source/control 
$z$ is placed in the exterior domain $\Omc$, disjoint from $\Om$. This is not possible
using the classical models. The authors in \cite{HAntil_RKhatri_MWarma_2018a} have
recently introduced the notion of exterior optimal control with elliptic fractional PDEs as 
constraints. The current paper develops a complete theoretical framework for the parabolic case. 
The paper \cite{HAntil_RKhatri_MWarma_2018a} has been inspired by the work of M. Warma \cite{warma2018approximate} where the author has shown that the classical notion of 
controllability (for fractional PDEs) from the boundary does not make sense and therefore 
it must be replaced by a control that is localized outside the open set where the PDE is
solved. For completeness, we would like to mention that the authors have recently considered 
the case where the source/control is located in the interior \cite{antil2017optimal}, see also \cite{antil2018b,antil2016optimal} for the case when the source/control is the diffusion coefficient. We also mention the works on the interior control in case of the so-called spectral fractional Laplacian \cite{AnPfWa2017, antil2017optimal} and for boundary control see \cite{antil2017fractional}. We also mention some interesting but (not directly related) works on fractional 
Calder\'on type inverse problems \cite{GLX17,LL17,RS17}. Notice that fractional operators
further provide flexibility to approximate arbitrary functions \cite{CDV18, DSV16, Grub, Kry18}.
 
The key difficulties and novelties of this paper are as follows:
\begin{enumerate}[(i)]
\item {\bf Nonlocal diffusion operator and exterior conditions.} The fractional Laplacian 
      $(-\Delta)^s$ is a   
      nonlocal operator and its evaluation at a point requires information over the entire        
      $\RR^N$. In addition, $(-\Delta)^s u$ may be nonsmooth even if $u$ is smooth
      (see e.g. \cite[Remark 7.2]{RS-DP}). Moreover, we do not have the notion of boundary conditions, 
      but the exterior conditions on $\Omc$. 

\item {\bf Nonlocal normal derivative.} 
      $\mathcal{N}_s u$ is the nonlocal normal derivative
      of $u$. This can be thought of as a restricted fractional Laplacian in $\Omc$. It is a very 
      difficult object to handle both at the continuous and at the discrete levels. Indeed, 
      the best known regularity result for $\mathcal{N}_s$ is given in Lemma~\ref{lem:Nmap}
      which says that globally $\mathcal{N}_s u \in L^2(\Omc)$ for $u \in W^{s,2}(\RR^N)$. Higher
      regularity results are currently unknown.

          \item {\bf Approximation of Dirichlet problem by Robin.}
                In case of the parabolic Dirichlet problem \eqref{eq:dcp}, it is 
                imperative to deal 
      with $\mathcal{N}_s$. Indeed, we need to approximate the very-weak solution to the 
      parabolic Dirichlet problem \eqref{eq:Sd} which requires computing $\mathcal{N}_s$
      of the test functions (see \eqref{eq:vw_d}). Moreover, the optimality system 
      for the parabolic Dirichlet control problem  \eqref{eq:dcp} requires an approximation of
      the $\mathcal{N}_s$ of the adjoint variable (see \eqref{eq:dirfOC}).
      We circumvent the first difficulty by approximating the parabolic Dirichlet problem
      \eqref{eq:Sd} by a parabolic Robin problem. We also prove a rate of convergence for
      this approximation. Under this new setup, the first order optimality conditions do
      not require an approximation of the $\mathcal{N}_s$ of the adjoint variable.

\item {\bf Weak and very-weak solutions.} We study the notion of weak-solutions to the 
      parabolic Dirichlet problem \eqref{eq:Sd} which require a higher regularity on the 
      datum $z \in H^1((0,T);W^{s,2}(\Omc))$. 
      Since for the control problem \eqref{eq:dcp} we only assume that 
      $Z_D := L^2((0,T);L^2(\Omc))$, therefore we also develop an even weaker notion of 
      solutions to \eqref{eq:Sd}. We call it very-weak solutions.       
      We also develop the notion of weak-solutions to the
      Robin problem \eqref{eq:Sn} and prove their existence and uniqueness. 
      
\item {\bf Optimal control problems.} We establish the well-posedness of solutions to 
      both parabolic Dirichlet and the parabolic Robin control problems. 
\end{enumerate}

Models with fractional derivatives are becoming increasing popular which can be attributed to their role in many applications. These models appear in (but not limited to) image denoising, image segmentation and phase field modeling \cite{antil2017spectral,HAntil_SBartels_GDogan_2019a,HAntil_CNRautenberg_2018b}; data analysis and fractional diffusion maps \cite{antil2018fractional}; magnetotellurics (geophysics) \cite{CWeiss_BvBWaanders_HAntil_2018a}. 

In many realistic applications, the source/control is placed outside the domain where a PDE is fulfilled. Some examples of problems where this \emph{may be of relevance} are:
(a) Magnetic drug delivery: the drug with ferromagnetic particles is injected in the
body and external magnetic field is used to steer it to a desired location \cite{HAntil_RHNochetto_PVenegas_2018a, HAntil_RHNochetto_PVenegas_2018b,lubbe1996clinical}; 
(b) Acoustic testing: the aerospace structures are subjected to sound from the loudspeakers
\cite{larkin1999direct}.

The rest of the paper is organized as follows. We begin with Section~\ref{s:not} which introduces the notations and some preliminary results. The content of this section is well-known. Our main work starts from Section~\ref{s:state} where we first study the notion of weak and very weak solutions to the parabolic Dirichlet problem in Section~\ref{s:dbvp}. This is followed by the notion of weak solution to the Robin problem in Section~\ref{s:rbvp}. The emphasis of Section~\ref{s:exctrp} is on the parabolic Dirichlet and the parabolic Robin optimal control problems. In Section~\ref{s:extapp}, we discuss the approximation of the parabolic Dirichlet problem and parabolic Dirichlet control problem by the parabolic Robin ones. Finally, 
in Section~\ref{s:numerics} we discuss the numerical approximations of all the problems. The numerical experiments confirm our theoretical estimates. The experiments on the control/source identification problem illustrate the strength of nonlocal approach over the local ones.

\section{Notation and Preliminaries}\label{s:not}

The purpose of this section is to introduce the notations and some
preliminary results. The results of this section are well-known. 
We follow the notation from \cite{HAntil_RKhatri_MWarma_2018a, warma2018approximate}.
Unless otherwise stated, $\Om \subset\RR^N$ ($N \ge 1$) is a bounded  open set and $0 < s < 1$. Let 
 \[
    W^{s,2}(\Om) := \left\{ u \in L^2(\Om) \;:\; 
            \int_\Om\int_\Om \frac{|u(x)-u(y)|^2}{|x-y|^{N+2s}}\;dxdy < \infty \right\} ,
 \]
and we endow it with the norm defined by

 \[
    \|u\|_{W^{s,2}(\Om)} := \left(\int_\Om |u|^2\;dx 
        + \int_\Om\int_\Om  \frac{|u(x)-u(y)|^2}{|x-y|^{N+2s}}\;dxdy \right)^{\frac12}.      
 \]
 
In order to study the Dirichlet problem \eqref{eq:Sd} we also need to define 
 \[
    W^{s,2}_0(\overline\Om) := \left\{ u \in W^{s,2}(\RR^N) \;:\; u = 0 \mbox{ in } 
            \RR^N\setminus\Om \right\} . 
 \]
In this case 
\begin{align*}
\|u\|_{W_0^{s,2}(\bOm)}:=\left(\int_{\RR^N}\int_{\RR^N}\frac{|u(x)-u(y)|^2}{|x-y|^{N+2s}}\;dxdy\right)^{\frac 12}
\end{align*}
defines an equivalent norm on $W_0^{s,2}(\bOm)$. 

The dual spaces of $W^{s,2}(\RR^N)$ and $W_0^{s,2}(\overline\Om)$ are denoted by $W^{-s,2}(\RR^N)$ and $W^{-s,2}(\overline\Om)$, respectively. 
Moreover, $\langle\cdot,\cdot\rangle$ shall denote their duality pairing whenever it is clear from the context. 
  
The local fractional order Sobolev space is defined as
 \begin{equation}\label{eq:Ws2loc}
    W^{s,2}_{\rm loc}(\Omc) := \left\{ u \in L^2(\Omc) \;:\; u\varphi \in W^{s,2}(\Omc), 
         \ \forall \ \varphi \in \mathcal{D}(\Omc) \right\}.  
 \end{equation}
 
 If $s=1$, then we shall denote $W^{1,2}(\Omega)$ by $H^1(\Omega)$.

Finally, we are ready to introduce the fractional Laplace operator. 
We set 
\begin{equation*}
\mathbb{L}_s^{1}(\RR^N):=\left\{u:\RR^N\rightarrow
\mathbb{R}\;\mbox{
measurable, }\;\int_{\RR^N}\frac{|u(x)|}{(1+|x|)^{N+2s}}%
\;dx<\infty \right\} ,
\end{equation*}%
and for $u\in \mathbb{L}_s^{1}(\RR^N)$ and $\varepsilon >0$, we let
\begin{equation*}
(-\Delta )_{\varepsilon }^{s}u(x)=C_{N,s}\int_{\{y\in \RR^N,|y-x|>\varepsilon \}}
\frac{u(x)-u(y)}{|x-y|^{N+2s}}dy,\;\;x\in\RR^N,
\end{equation*}%
where the normalized constant $C_{N,s}$ is given by
\begin{equation}\label{CN}
C_{N,s}:=\frac{s2^{2s}\Gamma\left(\frac{2s+N}{2}\right)}{\pi^{\frac
N2}\Gamma(1-s)},
\end{equation}%
and $\Gamma $ is the usual Euler Gamma function (see, e.g. \cite%
{BCF,Caf3,Caf1,Caf2,NPV,War-DN1,War}). Then the {\bf fractional Laplacian} 
$(-\Delta )^{s}$ is defined for $u\in \mathbb{L}_s^{1}(\RR^N)$  by the formula

\begin{align}
(-\Delta )^{s}u(x)=C_{N,s}\mbox{P.V.}\int_{\RR^N}\frac{u(x)-u(y)}{|x-y|^{N+2s}}dy 
=\lim_{\varepsilon \downarrow 0}(-\Delta )_{\varepsilon
}^{s}u(x),\;\;x\in\RR^N,\label{eq11}
\end{align}%
provided that the limit exists. We remark that it has been shown in \cite[Proposition 2.2]{BPS} that for $u \in \mathcal{D}(\Om)$, we have 
that
 \[
    \lim_{s\uparrow 1^-}\int_{\RR^N} u (-\Delta)^su\;dx 
        = \int_{\RR^N} |\nabla u|^2 dx 
        = - \int_{\RR^N} u \Delta u\;dx 
        = - \int_{\Om} u \Delta u\;dx.
 \]
This is where the constant $C_{N,s}$ plays a crucial role.

Now, we define the operator $(-\Delta)_D^s$ in $L^2(\Omega)$ as follows. 

\begin{equation}
D((-\Delta)_D^s)=\Big\{u|_{\Om},\; u\in W_0^{s,2}(\bOm):\; (-\Delta)^su\in L^2(\Omega)\Big\},\;\;(-\Delta)_D^s(u|_{\Om})=(-\Delta)^su\;\mbox{ in }\;\Omega.
\end{equation} 
Then $(-\Delta)_D^s$ is the realization in $L^2(\Om)$ of the fractional Laplace operator $(-\Delta)^s$ with the Dirichlet exterior condition $u=0$ in $\RR^N\setminus\Omega$. The following result is well-known (see e.g. \cite{BWZ,SV2}).  

\begin{proposition}\label{semigroup}
The operator $(-\Delta)_D^s$ has compact resolvent and $-(-\Delta)_D^s$ generates  a strongly continuous semigroup $(e^{-t(-\Delta)_D^s})_{t\ge 0}$ on $L^2(\Omega)$.
\end{proposition}

Next, for $u \in W^{s,2}(\RR^N)$ we define the nonlocal normal derivative $\mathcal{N}_s$ as follows:
 \begin{align}\label{NLND}
    \mathcal{N}_s u(x) := C_{N,s} \int_\Om \frac{u(x)-u(y)}{|x-y|^{N+2s}}\;dy, 
            \quad x \in \RR^N \setminus \overline\Om . 
 \end{align}
 We shall call $\mathcal N_s$ the {\em interaction operator}. Notice that the origin of the  
 term ``interaction" goes back to \cite{QDu_MGunzburger_RBLehoucq_KZhou_2013a}.
Clearly $\mathcal{N}_s$ is a nonlocal operator and it is well defined on $W^{s,2}(\RR^N)$
as we discuss next.

 \begin{lemma}\label{lem:Nmap}
  The interaction operator $\mathcal{N}_s$ maps continuously $W^{s,2}(\RR^N)$ into
 $ W^{s,2}_{\rm loc}(\RR^N\setminus\Om)$. 
  As a result, if  $u \in W^{s,2}(\RR^N)$, then $\mathcal{N}_s u \in L^2(\RR^N\setminus\Om)$.
 \end{lemma}

Despite the fact that $\mathcal{N}_s$ is defined on $\RR^N \setminus \Om$, it
is still known as the ``normal" derivative. This is due to its similarity with the classical normal derivative \cite[Proposition~2.2]{HAntil_RKhatri_MWarma_2018a}. We conclude this section by stating the integration by parts formula for the fractional Laplacian (see e.g. \cite{SDipierro_XRosOton_EValdinoci_2017a}). 

 \begin{proposition}[\bf The integration by parts formula for $(-\Delta)^s$]
 \label{prop:prop}
 Let $u \in W^{s,2}(\RR^N)$ be such that $(-\Delta)^su \in L^2(\Omega)$. 
 Then for every $v\in W^{s,2}(\RR^N)$ we have that
      \begin{align}\label{Int-Part}
       \frac{C_{N,s}}{2} 
        \int\int_{\RR^{2N}\setminus(\RR^N\setminus\Om)^2} 
         \frac{(u(x)-u(y))(v(x)-v(y))}{|x-y|^{N+2s}} \;dxdy 
        = \int_\Om v(-\Delta)^s u\;dx + \int_{\RR^N\setminus\Om} v\mathcal{N}_s u\;dx ,
      \end{align}
      where $\RR^{2N}\setminus(\RR^N\setminus\Om)^2
       := (\Om\times\Om)\cup(\Om\times(\RR^N\setminus\Om))\cup((\RR^N\setminus\Om)\times\Om)$.       
 \end{proposition}

\section{The parabolic state equations}\label{s:state}
Before analyzing the optimal control problems \eqref{eq:dcp} and \eqref{eq:ncp}, for a given
function $z$, we shall focus on the Dirichlet \eqref{eq:Sd} and Robin \eqref{eq:Sn} exterior
value problems. We shall assume that $\Om$ is a bounded domain with a Lipschitz continuous boundary.

\subsection{The parabolic Dirichlet problem for the fractional Laplacian}\label{s:dbvp}

Let us consider the following auxiliary problem at first 

\begin{equation}\label{eq:Sd_mod}
 \begin{cases}
    \partial_t w + (-\Delta)^s w &= f \quad \mbox{in } Q, \\
                w &= 0           \quad \mbox{in } \Sigma , \\
                w(0,\cdot) & = 0 \quad \mbox{in } \Om , 
 \end{cases}                
 \end{equation}
i.e., a fractional parabolic equation with nonzero right-hand-side 
but zero exterior condition. Notice that \eqref{eq:Sd_mod} can be rewritten as the following Cauchy problem:

\begin{equation}\label{CP}
\begin{cases}
\partial_tw+(-\Delta)_D^sw=f\quad &\mbox{ in } Q,\\
w(0,\cdot)=0&\mbox{ in }\;\Omega.
\end{cases}
\end{equation}

We next state the notion of a weak solution
to \eqref{eq:Sd_mod}:

 \begin{definition}[\bf Weak solution: homogeneous Dirichlet case] 
 \label{def:weak_d}
    Let $f \in L^2((0,T);W^{-s,2}(\overline\Om))$. 
    A $w \in \mathbb{U}_0 := L^2((0,T);W^{s,2}_0(\overline\Om)) \cap H^1((0,T);W^{-s,2}(\overline\Om))$ 
    is said to be a weak solution to \eqref{eq:Sd_mod} if 
    \[
     \langle \partial_t w , v \rangle +
     \frac{C_{N,s}}{2} 
        \int_{\RR^N}\int_{\RR^{N}} 
         \frac{(w(x)-w(y))(v(x)-v(y))}{|x-y|^{N+2s}} \;dxdy 
         = \langle f , v \rangle , 
    \]
    for every $v \in W^{s,2}_0(\overline\Om)$ and almost every $t \in (0,T)$. 
 \end{definition}
 
 \begin{remark}\label{rem:cont}
 {\rm
     A weak solution to \eqref{eq:Sd_mod} belongs to $C([0,T],L^2(\Om))$
     (see \cite[Remark~9]{TLeonori_IPeral_APrimo_FSoria_2015a} for details).
 }
 \end{remark}

 The existence and uniqueness of solution to \eqref{eq:Sd_mod} was shown in 
 \cite[Theorem~26]{TLeonori_IPeral_APrimo_FSoria_2015a}.

 \begin{proposition}[\bf Weak solution to \eqref{eq:Sd_mod}]
 \label{prop:weak_Dir}
  Let  $f \in L^2((0,T);W^{-s,2}(\overline\Om))$. Then there exists a unique weak 
  solution $w \in \mathbb{U}_0$ to \eqref{eq:Sd_mod} in the sense of 
  Definition~\ref{def:weak_d} given by
  \begin{align*}
  u(t,x)=\int_0^t e^{-(t-\tau)(-\Delta)_D^s}f(\tau,x)\;d\tau,
  \end{align*}
  where $(e^{-t(-\Delta)_D^s})_{t\ge 0}$ is the semigroup mentioned in Proposition \ref{semigroup}.
  In addition there is a constant $C>0$ such that
  \begin{align}\label{Es-DS-0}
   \|w\|_{\mathbb{U}_0} 
    \le C \|f\|_{L^2((0,T);W^{-s,2}(\overline\Om))}  .
  \end{align}
\end{proposition}

 We next introduce the notion of weak solution to our nonhomogeneous problem 
  \eqref{eq:Sd}. Notice the higher regularity requirement on the datum $z$.
 
 \begin{definition}[\bf Weak solution: nonhomogenous Dirichlet case] 
 \label{def:weak_d_1}
    Let the function \\$z \in H^1((0,T);W^{s,2}(\RR^N \setminus\Om))$ and $\tilde{z} \in H^1((0,T);W^{s,2}(\RR^N))$
    be such that $\tilde{z}|_{\Omc} = z$. Then a 
    $u \in \mathbb{U} := L^2((0,T);W^{s,2}(\RR^N)) \cap H^1((0,T);W^{-s,2}(\overline\Om))$
    is said to be a weak solution to \eqref{eq:Sd} if $u-\tilde{z} \in \mathbb{U}_0$
    and 
    \[
     \langle \partial_t u , v \rangle +
     \frac{C_{N,s}}{2} 
        \int_{\RR^N}\int_{\RR^{N}} 
         \frac{(u(x)-u(y))(v(x)-v(y))}{|x-y|^{N+2s}} \;dxdy 
         = 0 , 
    \]
    for every $v \in W^{s,2}_0(\overline\Om)$ and almost every $t \in (0,T)$. 
 \end{definition}

 Towards this end, we show the well-posedness of \eqref{eq:Sd}.

 \begin{theorem}[\bf Weak solution to \eqref{eq:Sd}]\label{Theo35}
  Let $z \in H^1((0,T);W^{s,2}(\RR^N \setminus\Om))$ be given.
  Then there exists a unique weak solution 
  $u \in \mathbb{U}$ to \eqref{eq:Sd} in the sense of 
  Definition~\ref{def:weak_d_1}. 
   In addition there is 
  a  constant $C>0$ such that
  \begin{align}\label{Es-DS}
   \|u\|_{\mathbb{U}} 
    \le C \|z\|_{H^1((0,T);W^{s,2}(\RR^N \setminus\Om))} .
  \end{align}
 \end{theorem}
 \begin{proof}
  Before we proceed with the proof, we need some preparation. Let us first
  assume that $z$ only depends on the spatial variable $x$ and consider the 
  $s$-Harmonic extension $\tilde{z} \in W^{s,2}(\RR^N)$ of 
  $z \in W^{s,2}(\RR^N\setminus\Om)$ that solves the Dirichlet problem
  \begin{equation}\label{eq:vwellip}
  \begin{cases}
      (-\Delta)^s \tilde{z} = 0 \quad &\mbox{in } \Om , \\
                  \tilde{z} = z \quad &\mbox{in } \Omc ,
  \end{cases}
  \end{equation}
  in a weak sense, i.e., given $z \in W^{s,2}(\RR^N\setminus\Om)$ there exists a unique 
  $\tilde{z} \in W^{s,2}(\RR^N)$, such that $\tilde{z}|_{\Omc} = z$, $\tilde z$ solves 
  \eqref{eq:vwellip} in the  sense that
  \[
      \frac{C_{N,s}}{2} 
        \int_{\RR^N}\int_{\RR^{N}} 
         \frac{(\tilde{z}(x)-\tilde{z}(y))(v(x)-v(y))}{|x-y|^{N+2s}} \;dxdy  
         = 0  
      \quad \mbox{for all } v \in W^{s,2}_0(\overline\Om),
  \]
and there is a constant $C > 0$ such that 
  \begin{equation}\label{E-DP}
      \|\tilde{z}\|_{W^{s,2}(\RR^N)} \le C \|z\|_{W^{s,2}(\RR^N\setminus\Om)} .
  \end{equation}
  The existence of a weak solution to \eqref{eq:vwellip} and the continuous dependence on data $z$ have been shown 
  in \cite{GGrubb_2015a}, see also \cite{ghosh2016calder,MIVisik_GIEskin_1965a}.
  When $z$ is a function of $(x,t)$ then it follows from the above arguments that 
  if $z \in L^2((0,T);W^{s,2}(\Omc))$ then $\tilde{z} \in L^2((0,T);W^{s,2}(\RR^N))$. 
  On the other hand, if $\partial_t z \in L^2((0,T);W^{s,2}(\Omc))$ then 
  $\partial_t \tilde{z} \in L^2((0,T);W^{s,2}(\RR^N))$.
  
 Now we show the existence of a unique solution to \eqref{eq:Sd} using a lifting
  argument. We define $w := u - \tilde{z}$. Then $w|_{\Omc} = 0$. Moreover,
  a simple calculation shows that $w$ fulfills 
  \begin{equation}\label{eq:Sd_1}
  \begin{cases}
    \partial_t w + (-\Delta)^s w = -\partial_t \tilde{z}   \quad& \mbox{in } Q, \\
                w = 0           \quad &\mbox{in } \Sigma ,  \\
                w(0,\cdot)  = 0 \quad& \mbox{in } \Om .
  \end{cases}                
  \end{equation}
  Since we have assumed that $\partial_t z \in L^2((0,T);W^{s,2}(\Omc))$, therefore from the above
  discussion we have that $\partial_t\tilde{z} \in L^2((0,T);W^{s,2}(\RR^N))$. Hence, using
  Proposition~\ref{eq:Sd_mod}, we get that there exists a unique $w \in \mathbb{U}_0$ 
  solving 
  \eqref{eq:Sd_1}. Thus the unique solution $u \in \mathbb{U}$ is given by $u=w+\tilde z$. It remains to show the estimate \eqref{Es-DS}. Firstly, since $w=0$ in $\RR^N\setminus \Omega$, it follows from \eqref{Es-DS} that there is a constant $C>0$ such that
  \begin{align}\label{A1}
  \|w\|_{\mathbb U}=\|w\|_{\mathbb U_0}\le C\|\partial_t\tilde z\|_{L^2((0,T);W^{-s,2}(\bOm)}.
  \end{align}
Secondly, it follows from \eqref{E-DP} that there is a constant $C>0$ such that
\begin{align}\label{A2}
  \|\tilde{z}\|_{L^2((0,T);W^{s,2}(\RR^N))} \le C \|z\|_{L^2((0,T);W^{s,2}(\RR^N\setminus\Om))} .
\end{align}  
Thirdly, using \eqref{A1} and \eqref{A2} we get that there is a constant $C>0$ such that
\begin{align}\label{A3}
\|u\|_{\mathbb U}&=\|w+\tilde z\|_{\mathbb U}\le \|w\|_{\mathbb U}+\|\tilde z\|_{\mathbb U}\notag\\
&\le C\left(\|\partial_t\tilde z\|_{L^2((0,T);W^{-s,2}(\bOm))}+ \|z\|_{L^2((0,T);W^{s,2}(\RR^N\setminus\Om))} +\|\tilde z\|_{H^1((0,T);W^{-s,2}(\bOm))}\right)\notag\\
&\le C\left(\|\partial_t\tilde z\|_{L^2((0,T);W^{-s,2}(\bOm))}+ \|z\|_{L^2((0,T);W^{s,2}(\RR^N\setminus\Om))} +\|\tilde z\|_{L^2((0,T);W^{-s,2}(\bOm))}\right).
\end{align}
Since $\tilde z\in L^2((0,T);W^{s,2}(\RR^N))$, then using \eqref{E-DP}, we get that

\begin{align}\label{mw1}
\|\tilde z\|_{L^2((0,T);W^{-s,2}(\bOm))}\le &C\|\tilde z\|_{L^2((0,T);W^{-s,2}(\RR^N))}\le C\|\tilde z\|_{L^2((0,T);W^{s,2}(\RR^N))}\notag\\
\le &C\|z\|_{L^2((0,T);W^{s,2}(\RR^N\setminus\Om))}.
\end{align}

Note that $\partial_t\tilde z$ is a solution of the Dirichlet problem \eqref{eq:vwellip} with $z$ replaced with $\partial_tz$. This shows that $\partial_t\tilde z\in L^2((0,T);W^{s,2}(\RR^N))$. Hence, using  \eqref{E-DP} again, we obtain that

\begin{align}\label{mw2}
\|\partial_t\tilde z\|_{L^2((0,T);W^{-s,2}(\bOm))}\le &C\|\partial_t\tilde z\|_{L^2((0,T);W^{-s,2}(\RR^N))}\le C\|\partial_t\tilde z\|_{L^2((0,T);W^{s,2}(\RR^N))}\notag\\
\le &C\|\partial_t z\|_{L^2((0,T);W^{s,2}(\RR^N\setminus\Om))}.
\end{align}
Combining \eqref{mw1} and \eqref{mw2}, we get from \eqref{A3} that
\begin{align*}
\|u\|_{\mathbb U}\le C\left( \|z\|_{L^2((0,T);W^{s,2}(\RR^N\setminus\Om))}+\|\partial_t z\|_{L^2((0,T);W^{s,2}(\RR^N\setminus\Om))}\right).
\end{align*}
We have shown \eqref{Es-DS} and  the proof is finished.
\end{proof}

\begin{remark}
{\em Let $(\varphi_n)_{n\in\NN}$ be the orthonormal basis of eigenfunctions of $(-\Delta)_D^s$ associated with the eigenvalues $(\lambda_n)_{n\in\NN}$. If in Theorem \ref{Theo35}, one assumes that $z\in H^1((0,T);W^{s,2}(\RR^N\setminus\Omega))$ with $z(0,\cdot)=0$, then it has been shown in \cite[Theorem 18]{warma2018approximate} that the unique weak solution $u$ of  \eqref{eq:Sd} is given by
\begin{align*}
u(t,x)=-\sum_{n=1}^\infty\left(\int_0^t\left(z(\cdot,t-\tau),\mathcal N_s\varphi_n\right)_{L^2(\RR^N\setminus\Omega} e^{-\lambda_n\tau}\;d\tau\right)\varphi_n(x) . 
\end{align*}
}
\end{remark}

Our next goal is to reduce the regularity requirements on the datum $z$ 
in both space and time. We shall call the resulting solution $u$ as 
very-weak solution.

 \begin{definition}[\bf Very-weak solution: nonhomogenous Dirichlet case] 
 \label{def:vweak_d}
    Let the function $z \in L^2((0,T);L^2(\RR^N\setminus\Om))$. A  $u \in L^2((0,T);L^2(\RR^N))$ 
    is said to be a very-weak solution to \eqref{eq:Sd} if the identity
    \begin{equation}\label{eq:vw_d}
      \int_{Q} 
         u \left(-\partial_t v + (-\Delta)^s v \right)\;dxdt
         = - \int_{\Sigma} z \mathcal{N}_s v\;dxdt ,
    \end{equation}
   holds for every $v \in L^2((0,T);V) \cap H^1((0,T);L^2(\Om))$ with $v(T,\cdot) = 0$, where 
   $V := \{v \in W^{s,2}_0(\overline\Om) \;:\; (-\Delta)^s v \in L^2(\Om) \}$. 
 \end{definition}
The following result shows the existence and uniqueness of a very-weak solution 
to \eqref{eq:Sd} in the sense of Definition~\ref{def:vweak_d}. We will prove this
result by using a duality argument (see e.g.  \cite{WGong_MHinze_ZZhou_2016a} for the classical case $s=1$).

 \begin{theorem}\label{thm:vwdexist}
  Let $z \in L^2((0,T);L^2(\RR^N\setminus\Om))$. Then there exists a unique 
  very-weak solution $u$ to \eqref{eq:Sd} according to Definition~\ref{def:vweak_d} 
  that fulfills
  \begin{align}\label{VWS_EST}
   \|u\|_{L^2((0,T);L^2(\Om))} \le C \|z\|_{L^2((0,T);L^2(\RR^N \setminus\Om))} 
  \end{align}
for a constant $C >0$. 
In addition, if $z\in H^1((0,T);W^{s,2}(\Omc))$, then the following assertions hold.
\begin{enumerate}
\item Every weak solution of \eqref{eq:Sd} is also a very-weak solution.
\item Every very-weak solution of \eqref{eq:Sd} that belongs to 
      $\mathbb{U}$ is also a weak solution.
\end{enumerate}
 \end{theorem}
 
 \begin{proof}  
  For a given 
  $\zeta \in L^2((0,T);L^2(\Om))$, we begin by considering the following 
  ``dual" problem 
  \begin{equation}\label{eq:Sd_dual}
  \begin{cases}
    -\partial_t v + (-\Delta)^s v &= \zeta \quad \mbox{in } Q, \\
                v &= 0           \quad \mbox{in } \Sigma ,  \\
                v(T,\cdot) & = 0 \quad \mbox{in } \Om . 
 \end{cases}                
 \end{equation} 
Using semigroup theory as in  Proposition~\ref{prop:weak_Dir}, one can easily deduce that the problem \eqref{eq:Sd_dual} has a unique weak solution 
 $v \in \mathbb{U}_0$. 
 
 Since $v \in L^2((0,T);W^{s,2}_0(\overline\Om))$, owing to Lemma~\ref{lem:Nmap}, we 
 have that $\mathcal{N}_s v \in L^2((0,T);L^2(\RR^N\setminus\Om))$. 
 Towards, this end we define the mapping 
 \begin{equation*}
     \mathcal{M} : L^2((0,T);L^2(\Om)) \rightarrow L^2((0,T);L^2(\RR^N\setminus\Om)), \qquad
                           \zeta \mapsto \mathcal{M} \zeta := -\mathcal{N}_s v .
 \end{equation*}
 We notice that $\mathcal{M}$ is linear and continuous because 
 \[
     \|\mathcal{M}\zeta\|_{L^2((0,T);L^2(\RR^N\setminus\Om))} 
      = \|\mathcal{N}_s v\|_{L^2((0,T);L^2(\RR^N\setminus\Om))} 
      \le C \|v\|_{L^2((0,T);W^{s,2}_0(\overline\Om))}
      \le C\|\zeta\|_{L^2((0,T);L^2(\Om))} . 
 \] 
 
 Let $u := \mathcal{M}^* z$, then we have  
 \[
     \int_Q u \zeta \; dxdt 
      = \int_Q u \left( -\partial_t v + (-\Delta)^s v \right) \;dxdt
      = \int_Q (\mathcal{M}^* z) \zeta \; dxdt
      = -\int_\Sigma z \mathcal{N}_s v \; dxdt.
 \]
 We have constructed a unique $u \in L^2((0,T);L^2(\RR^N))$ that 
 solves \eqref{eq:vw_d}. Finally, we notice that 
 \[
     \left|\int_Q u \zeta \; dxdt \right| 
         \le \|z\|_{L^2((0,T);L^2(\Om))} \|\mathcal{N}_s v\|_{L^2((0,T);L^2(\RR^N\setminus\Om))}
         \le C \|z\|_{L^2((0,T);L^2(\Om))} \|\zeta\|_{L^2((0,T);L^2(\Om))} .
 \]
 Dividing both sides by $\|\zeta\|_{L^2((0,T);L^2(\Om))}$ and taking the supremum over
 $\zeta \in L^2((0,T);L^2(\Om))$ we obtain \eqref{VWS_EST}.
 
 Next we prove the last two assertions of the theorem. Assume that  
 $z\in H^1((0,T);W^{s,2}(\Omc))$.
  
 (a) Let $u \in \mathbb{U} \hookrightarrow L^2((0,T);L^2(\RR^N))$ 
 be a weak solution to \eqref{eq:Sd}. It follows from 
 the definition that $u = z$ on $\Omc$ and
 \begin{align}\label{e1}
     \langle \partial_t u , v \rangle +
     \frac{C_{N,s}}{2} 
        \int_{\RR^N}\int_{\RR^{N}} 
         \frac{(u(x)-u(y))(v(x)-v(y))}{|x-y|^{N+2s}} \;dxdy 
         = 0 , 
 \end{align}
 for every $v \in L^2((0,T);V) \cap H^1((0,T);L^2(\Om))$ and almost every $t \in (0,T)$.
 Since $v=0$ in $\RR^N\setminus\Omega$, we have that
\begin{align}\label{e2}
\int_{\RR^N}\int_{\RR^N}&\frac{(u(x)-u(y))(v(x)-v(y))}{|x-y|^{N+2s}}\;dxdy\notag\\
&=\int\int_{\RR^{2N}\setminus(\RR^N\setminus\Omega)^2}\frac{(u(x)-u(y))(v(x)-v(y))}{|x-y|^{N+2s}}\;dxdy.
\end{align}
Using \eqref{e1}, \eqref{e2}, the integration by parts formula \eqref{Int-Part} together with the fact that $u=z$ in $\RR^N\setminus\Omega$, we get that
\begin{align*}
\langle \partial_t u,v \rangle + 
\frac{C_{N,s}}{2}\int_{\RR^N}\int_{\RR^N}&\frac{(u(x)-u(y))(v(x)-v(y))}{|x-y|^{N+2s}}\;dxdy\\
&=0\\
&=\langle\partial_t u,v \rangle + \int_{\Omega}u(-\Delta)^sv\;dx+\int_{\RR^N\setminus\Omega}u\mathcal N_sv\;dx\\
&=\langle\partial_t u,v \rangle + \int_{\Omega}u(-\Delta)^sv\;dx+\int_{\RR^N\setminus\Omega}z\mathcal N_sv\;dx.
\end{align*}
Thus $u$ is a very-weak solution of \eqref{eq:Sd}. \\

 (b) Let $u$ be a very-weak solution to \eqref{eq:Sd} and assume that $u \in \mathbb{U}$. 
 We have that $u=z$ in $(0,T) \times \Omc$. Moreover, $z\in H^1((0,T);W^{s,2}(\Omc))$
 and if $\tilde{z} \in H^1((0,T);W^{s,2}(\RR^N))$ is such that $\tilde z|_{\RR^N\setminus\Omega}=z$, then clearly $u-\tilde{z} \in 
 \mathbb{U}_0$. Since $u$ is a very-weak solution to \eqref{eq:Sd}, then by 
 Definition~\ref{def:vweak_d}, for every $v \in L^2((0,T);V) \cap H^1((0,T);L^2(\Om))$
 with $v(T,\cdot) = 0$, we have that 
 \begin{align}\label{e3}
  \int_{Q}u(-\partial_t v + (-\Delta)^s v)\;dx
  =-\int_{\Sigma}z\mathcal N_sv\;dx.
 \end{align}
 Since $u \in \mathbb{U}$, $v = 0$ on $\Omc$, using the integration by parts formula
 \eqref{Int-Part} we get that 
 \begin{align}\label{e4}
\int_0^T \langle \partial_t u, v \rangle\;dt &+ 
\int_0^T \int_{\RR^N}\int_{\RR^N}\frac{(u(x)-u(y))(v(x)-v(y))}{|x-y|^{N+2s}}\;dxdydt\notag\\
&= \int_0^T\langle \partial_t u, v \rangle \;dt +  \int_0^T\int\int_{\RR^{2N}\setminus(\RR^N\setminus\Omega)^2}\frac{(u(x)-u(y))(v(x)-v(y))}{|x-y|^{N+2s}}\;dxdydt\notag\\
&=\int_{Q}u(\partial_t v + (-\Delta)^sv)\;dxdt+\int_{\Sigma}u\mathcal N_sv\;dx\notag\\
&=\int_{Q}u(\partial_t v + (-\Delta)^sv)\;dxdt+\int_{\Sigma}z\mathcal N_sv\;dxdt. 
\end{align}
It then follows form \eqref{e3} and \eqref{e4} that for every 
$v \in L^2((0,T);V) \cap H^1((0,T);L^2(\Om))$ with $v(T,\cdot) = 0$ we have the identity
 \begin{equation}\label{eq:e5}
    \int_0^T\langle \partial_t u, v \rangle\;dt  + 
    \int_0^T
     \int_{\RR^N}\int_{\RR^N}\frac{(u(x)-u(y))(v(x)-v(y))}{|x-y|^{N+2s}}\;dxdydt = 0.
 \end{equation}
 Since $V$ is dense in $W_0^{s,2}(\overline\Om)$ and $L^2(\Om)$ is dense in 
 $W^{-s,2}(\overline\Om)$, it follows that \eqref{eq:e5} remains true for 
 $v \in \mathbb{U}_0$ with $v(T,\cdot) = 0$. Notice that for every 
 $t \in (0,T]$ we have that $v(t,\cdot) \in W^{s,2}_0(\overline\Om)$. As a result, we 
 have that the following pointwise formulation 
 \begin{equation}\label{eq:e6}
    \langle \partial_t u, v \rangle\;  +     
     \int_{\RR^N}\int_{\RR^N}\frac{(u(x)-u(y))(v(x)-v(y))}{|x-y|^{N+2s}}\;dxdy = 0,
 \end{equation}
 holds for every $v \in W_0^{s,2}(\overline\Om)$ which is independent on $t$. We have shown that $u$ is the 
 unique weak solution to  \eqref{eq:Sd} according to Definition~\ref{def:weak_d_1} and the 
 proof is complete.
 \end{proof}

\subsection{The parabolic Robin problem for the fractional Laplacian}\label{s:rbvp}

In this section, we shall consider the Robin problem 
\eqref{eq:Sn}. We begin by specifying the Sobolev space as introduced in
\cite{SDipierro_XRosOton_EValdinoci_2017a}. Here we follow the notation from
\cite{HAntil_RKhatri_MWarma_2018a}. For $g \in L^1(\Omc)$ fixed, we let 
    \begin{align*}
W_{\Omega,g}^{s,2}:=\Big\{u:\RR^N\to\RR\;\mbox{ measurable, }\,\|u\|_{W_{\Omega,g}^{s,2}}<\infty\Big\}, 
\end{align*}
where
\begin{align}\label{norm}
\|u\|_{W_{\Omega,g}^{s,2}}:=\left(\|u\|_{L^2(\Omega)}^2+\||g|^{\frac 12}u\|_{L^2(\RR^N\setminus\Omega)}^2+\int\int_{\RR^{2N}\setminus(\RR^N\setminus\Omega)^2}\frac{|u(x)-u(y)|^2}{|x-y|^{N+2s}}dxdy\right)^{\frac 12}.
\end{align}
 Let $\mu$ be the measure on $\RR^N\setminus\Omega$ given by $d\mu=|g|dx$.
With this setting, the norm in \eqref{norm} can be rewritten as 
\begin{align}\label{norm-e}
\|u\|_{W_{\Omega,g}^{s,2}}:=\left(\|u\|_{L^2(\Omega)}^2+\|u\|_{L^2(\RR^N\setminus\Omega,\mu)}^2+\int\int_{\RR^{2N}\setminus(\RR^N\setminus\Omega)^2}\frac{|u(x)-u(y)|^2}{|x-y|^{N+2s}}dxdy\right)^{\frac 12}.
\end{align}

If $g=0$, we shall let $W_{\Omega,0}^{s,2}=W_{\Omega}^{s,2}$.
The following result has been proved in \cite[Proposition 3.1]{SDipierro_XRosOton_EValdinoci_2017a}.

\begin{proposition}
Let $g\in L^1(\RR^N\setminus\Omega)$. Then $W_{\Omega,g}^{s,2}$ is a Hilbert space.
\end{proposition} 

Throughout the remainder of the paper, the measure $\mu$ is defined with $g$ replaced by 
$\kappa$. That is, $d\mu=\kappa dx$ (recall that $\kappa$ is assumed to be 
non-negative). We next state our notion of weak solution.

 \begin{definition}\label{def:weak_n}
  Let  $z \in L^2((0,T);L^2(\RR^N\setminus\Om,\mu))$. 
 A $u\in L^2((0,T);W_{\Om,\kappa}^{s,2})\cap H^1((0,T);( W_{\Om,\kappa}^{s,2})^\star)$ 
 is said to be a weak solution of \eqref{eq:Sn} if the identity
\begin{align}\label{we-so}
\langle \partial_t u, v \rangle + 
\int\int_{\RR^{2N}\setminus(\Omc)^2}&\frac{(u(x)-u(y))(v(x)-v(y))}{|x-y|^{N+2s}}\;dxdy+\int_{\Omc}\kappa uv\;dx\notag\\
&=
  \int_{\Omc}\kappa zv\;dx,
\end{align}
holds for every $v\in  W_{\Om,\kappa}^{s,2}$ and almost every $t \in (0,T)$. 
 \end{definition}
 
 Throughout the following, for $u,v\in W_{\Om, \kappa}^{s,2}$ we shall denote
 \begin{align*}
 \mathcal E(u,v):=\frac{C_{N,s}}{2}\int\int_{\RR^{2N}\setminus(\Omc)^2}&\frac{(u(x)-u(y))(v(x)-v(y))}{|x-y|^{N+2s}}\;dxdy+\int_{\Omc}\kappa uv\;dx.
 \end{align*}

Next we show the existence result.

\begin{theorem}\label{Theo-sol-Ro}
  Let $\kappa\in L^1(\Omc)\cap L^\infty(\Omc)$.
  Then for every $z\in L^2((0,T);L^2(\RR^N\setminus\Omega,\mu))$, 
  there exists a unique weak solution 
  $u\in L^2((0,T);W_{\Om,\kappa}^{s,2})\cap H^1((0,T);( W_{\Om, \kappa}^{s,2})^\star)$ 
  of \eqref{eq:Sn}. 
 \end{theorem}
 
 \begin{proof}
 We prove the result in several steps.
 
 {\bf Step 1}. Define the operator  $A$ in $L^2(\Omega)\times L^2(\Omc,\mu)$ as follows:
 \begin{align*}
 \begin{cases}
 D(A):=\Big\{(u,0):\;u \in W_{\Om, \kappa}^{s,2}, \;(-\Delta)^su\in L^2(\Omega),\; \mathcal N_su\in L^2(\Omc,\mu)\Big\},\\
 A(u,0)=\left(-(-\Delta)^su, -\mathcal N_s u-\kappa u\right).
 \end{cases}
 \end{align*}
 Let $(f,g)\in L^2(\Omega)\times L^2(\Omc,\mu)$. We  claim that $(u,0)\in D(A)$ with $-A(u,0)=(f,g)$ if and only if
 \begin{align}\label{RP-F}
 \mathcal E(u,v)=\int_{\Omega}fv\;dx+\int_{\Omc}gv\;d\mu,
 \end{align}
for all $v\in W_{\Om, \kappa}^{s,2}$.
 Indeed, we have that $(u,0)\in D(A)$ with $-A(u,0)=(f,g)$ if and only if $u$ is a weak solution of the elliptic problem 
 
 \begin{align}\label{RPE}
 \begin{cases}
 (-\Delta)^su=f\;\;\;&\mbox{ in }\;\Omega,\\
 \mathcal N_su+\kappa u=\kappa g\;\;\;&\mbox{ in }\;\Omc.
 \end{cases}
 \end{align}
 It has been shown in \cite{HAntil_RKhatri_MWarma_2018a} (see also \cite{CLR-MW}) that $u$ solves \eqref{RPE} if and only if \eqref{RP-F} holds and the claim is proved.
 
 {\bf Step 2}. Firstly, let $\lambda>0$ be a real number. We show that the operator $\lambda-A: D(A)\to L^2(\Omega)\times L^2(\Omc,\mu)$ is invertible. It is clear that for every $\lambda>0$ there is a constant $\alpha>0$ such that
 
 \begin{align}\label{28}
 \lambda\int_{\Omega}|u|^2\;dx+\mathcal E(u,u)\ge \alpha\|u\|_{W_{\Om, \kappa}^{s,2}}^2,
 \end{align}
 for all $u\in W_{\Om, \kappa}^{s,2}$.  Hence by Lax-Milgram Theorem, for every $(f,g)\in  L^2(\Omega)\times L^2(\Omc,\mu)$  there exists a unique $u\in W_{\Om, \kappa}^{s,2}$ such that
  
 \begin{align}\label{29}
 \lambda\int_{\Omega}uv\;dx+\mathcal E(u,v)=\int_{\Omega}fv\;dx+\int_{\Omc}gv\;d\mu,
 \end{align}
 for all $v\in W_{\Om, \kappa}^{s,2}$. By Step 1, this means  that there is a unique $u\in  W_{\Om, \kappa}^{s,2}$ with $(u,0)\in D(A)$ and
 \begin{align*}
 (\lambda-A)(u,0)=(\lambda u,0)-A(u,0)=(f,g).
 \end{align*}
 We have shown that $\lambda-A: D(A)\to  L^2(\Omega)\times L^2(\Omc,\mu)$ is a bijection for every $\lambda>0$.
 
 Secondly, assume now that $f\le 0$ a.e. in $\Omega$ and $g\le 0$ $\mu$-a.e. in $\Omc$. Let the function $(u,0):=(\lambda-A)^{-1}(f,g)$ and set $v:=u^+ :=\max\{u,0\}$. It follows from  \cite{War} that $u^+\in  W_{\Om, \kappa}^{s,2}$. Let $u^-:=\max\{-u,0\}$. Since
 \begin{align*}
 (u^-(x)-u^-(y))(u^+(x)-u^+(y))=&u^-(x)u^+(x)-u^-(x)u^+(y)-u^-(y)u^+(x)+u^-(y)u^+(y)\\
 =&-u^-(x)u^+(y)-u^-(y)u^+(x)\le 0,
 \end{align*}
 we have that $\mathcal E(u^-,u^+)\le 0$. Hence,
 \begin{align*}
 \mathcal E(u,v)=\mathcal E(u^+-u^-,u^+)=\mathcal E(u^+,u^+)-\mathcal E(u^-,u^+)\ge 0.
 \end{align*}
  Then by \eqref{29}, we have that
  \begin{align*}
  0\le \lambda\int_{\Omega}|u|^2\;dx+\mathcal E(u,u^+)=\int_{\Omega}fu^+\;dx+\int_{\Omc}gu^+\;d\mu\le 0.
  \end{align*}
By \eqref{28} this implies that $u^+=0$, that is, $u\le 0$ almost everywhere. We have shown that the resolvent $(\lambda-A)^{-1}$ is a positive operator. Since every positive linear operator is continuous (see e.g., \cite{ArNi}), we can deduce that $(\lambda-A)$ is in fact invertible.

Thirdly, we have in particular shown that the operator $A$ is closed since  $-A$ is the operator associated with the closed form $\mathcal E$. Hence $D(A)$  endowed with the graph norm is a Banach space  and by definition of $A$, we have that $D(A)\subset W_{\Om, \kappa}^{s,2}\times\{0\}$. Since both of these spaces are continuously embedded into $ L^2(\Omega)\times L^2(\Omc,\mu)$, we can deduce from the closed graph theorem that $D(A)$ is continuously embedded into $ W_{\Om, \kappa}^{s,2}\times\{0\}$.

{\bf Step 3}. Now since $L^2(\Omega)\times L^2(\Omc,\mu)$ is a Banach lattice with order continuous norm and by Step 2 the operator $A$ is resolvent positive,  it follows from \cite[Theorem 3.11.7]{ABHN} that $A$ generates a once integrated semigroup on $L^2(\Omega)\times L^2(\Omc,\mu)$. Hence using the theory of integrated semigroups and abstract Cauchy problems studied in \cite[Section 3.11]{ABHN} and proceeding as in \cite[Section 2]{Nittka}, we can deduce that for every $z\in L^2((0,T);L^2(\Omc,\mu))$, the problem \eqref{eq:Sn} has a unique weak solution. The proof is finished.
 \end{proof}
 
We conclude this section by showing that if $z$ is more regular in the time variable, then the existence of weak solutions can be easily proved without using the theory of integrated semigroups as in the proof of Theorem \ref{Theo-sol-Ro}.

 \begin{proposition}\label{pro-sol-Ro}
  Let $\kappa\in L^1(\Omc)\cap L^\infty(\Omc)$.
  Then for every $z\in H^1((0,T);L^2(\RR^N\setminus\Omega,\mu))$, 
  there exists a unique weak solution 
  $u\in L^2((0,T);W_{\Om,\kappa}^{s,2})\cap H^1((0,T);( W_{\Om, \kappa}^{s,2})^\star)$ 
  of \eqref{eq:Sn}. 
 \end{proposition}
 
 \begin{proof}
We proceed as in the proof of Theorem \ref{Theo35}. First, assume that $z$ does not depend on time and   $z\in L^2(\Omc,\mu)$. Let $\tilde z$ be the solution of the elliptic Robin problem

\begin{equation}\label{RP}
\begin{cases}
(-\Delta)^s\tilde z=0\quad&\mbox{ in } \Omega,\\
\mathcal N_s\tilde z+\kappa \tilde z=\kappa\tilde z &\mbox{ in } \RR^N\setminus\Omega,
\end{cases}
\end{equation}
 in the sense that $\tilde z\in W_{\Om,\kappa}^{s,2}$ and
 \begin{align}\label{form-F}
\int\int_{\RR^{2N}\setminus(\Omc)^2}\frac{(\tilde z(x)-\tilde z(y))(v(x)-v(y))}{|x-y|^{N+2s}}\;dxdy
  +\int_{\Omc}\kappa \tilde zv\;dx=\int_{\Omc}\kappa zv\;dx,
  \end{align}
  for every $v\in W_{\Om,\kappa}^{s,2}$. Under our assumptions, it has been shown in \cite{HAntil_RKhatri_MWarma_2018a}
 that \eqref{RP} has a  solution $\tilde z$.
  
  Next, assume that $z\in H^1((0,T); L^2(\Omc,\mu))$. Since in this case $\partial_t\tilde z$ will be a solution of \eqref{RP} with $z$ replaced by $\partial_tz$, then we can deduce that \eqref{RP} has a unique solution $\tilde z\in H^1((0,T); W_{\Om,\kappa}^{s,2})$. 
  
Consider the following parabolic problem
  \begin{equation}\label{C-RP}
\begin{cases}
\partial_t w+(-\Delta)^sw=-\partial_t\tilde z\quad &\mbox{ in } Q,\\
\mathcal N_sw+\kappa w=0 &\mbox{ in } \Sigma,\\
w(0,\cdot)=0&\mbox{ in }\Omega.
\end{cases}
\end{equation}
Let $(-\Delta)_R^s$ be the realization of $(-\Delta)^s$ with the zero Robin exterior condition $\mathcal N_sw+\kappa w=0$ in $\Omc$. 
Then the parabolic problem \eqref{C-RP} can be rewritten as the following Cauchy problem
\begin{equation*}
\begin{cases}
\partial_t w+(-\Delta)_R^sw=-\partial_t\tilde z\quad &\mbox{ in } Q,\\
w(0,\cdot)=0&\mbox{ in }\Omega.
\end{cases}
\end{equation*}
It has been shown in \cite{CLR-MW} that the operator $-(-\Delta)_R^s$ generates a strongly continuous semigroup $(e^{-t(-\Delta)_R^s})_{t\ge 0}$ in $L^2(\Omega)$.
Hence, using semigroup theory, we can deduce that \eqref{C-RP} has a unique weak solution $w$ that belongs to $L^2((0,T);W_{\Om,\kappa}^{s,2})\cap H^1((0,T);( W_{\Om, \kappa}^{s,2})^\star)$ and is given by
\begin{align*}
w(t,x)=-\int_0^te^{-(t-\tau)(-\Delta)_R^s}\partial_{\tau}\tilde z(\tau,x)\;dx.
\end{align*}
 It is clear that $u:=w+\tilde z$ is the unique weak solution of \eqref{eq:Sn}.  The proof is finished.
 \end{proof}

\section{Exterior Optimal Control Problems}\label{s:exctrp}

The purpose of this section is study the Dirichlet 
and the Robin  optimal control problems \eqref{eq:dcp}  and \eqref{eq:ncp}, respectively. These are the 
subjects of Sections~\ref{s:dir} and \ref{s:rob}, respectively.

\subsection{Fractional Dirichlet Exterior Control Problem}
\label{s:dir}

We begin by defining the function spaces $Z_D$ and $U_D$. We let 
    \[
        Z_D := L^2((0,T);L^2(\Omc)) , \quad U_D := L^2((0,T);L^2(\Om)) . 
    \]
Due to Theorem~\ref{thm:vwdexist}, the control-to-state (solution) map
    \[
        S : Z_D \rightarrow U_D , \quad z \mapsto Sz =: u,
    \]       
is well-defined, linear and continuous. Furthermore, for $z \in Z_D$, 
we have that $u := Sz \in L^2((0,T);L^2(\RR^N))$. Thus we can write the 
so-called reduced Dirichlet exterior parabolic optimal control problem as follows:
    \begin{equation}\label{eq:rpDir}
        \min_{z \in Z_{ad,D}} \mathcal{J}(z) := J(Sz) + \frac{\xi}{2}\|z\|^2_{Z_D} . 
    \end{equation}
Next, we state the well-posedness result for \eqref{eq:dcp} and equivalently 
\eqref{eq:rpDir}. 
\begin{theorem}\label{thm:docexist}
Let $Z_{ad,D}$ be a closed and convex subset of $Z_D$. Let either $\xi > 0$
or $Z_{ad,D}$ be bounded and let $J: U_D \rightarrow \RR$ be weakly lower-semicontinuous.
Then there exists a solution $\bar{z}$ to \eqref{eq:rpDir} and equivalently
\eqref{eq:dcp}. If either $J$ is convex and $\xi > 0$ or $J$ is strictly convex
and $\xi \ge 0$, then $\bar{z}$ is unique. 
\end{theorem}
\begin{proof}
The proof is based on the so-called direct method or the Weierstrass theorem 
\cite[Theorem~3.2.1]{HAttouch_GButtazzo_GMichaille_2014a}. We sketch the proof
here for completeness. For the functional $\mathcal{J} : Z_{ad,D} \rightarrow \mathbb{R}$, it is
possible to construct a minimizing sequence 
$\{z_n\}_{n\in\mathbb{N}}$ (see \cite[Theorem~3.2.1]{HAttouch_GButtazzo_GMichaille_2014a}) such that 
$\inf_{z\in Z_{ad,D}}\mathcal{J}(z) = \lim_{n\rightarrow\infty}
\mathcal{J}(z_n)$. If $\xi > 0$ or $Z_{ad,D} \subset Z_D$ is bounded, then
$\{z_n\}_{n\in\mathbb{N}}$ is a bounded sequence in $Z_D$ which is a Hilbert
space. As a result, we have that (up to a subsequence if necessary) 
$z_n \rightharpoonup \bar{z}$ (weak convergence) in $Z_D$ as $n\rightarrow
\infty$. Finally since $Z_{ad,D}$ is closed and convex,
hence is weakly closed, we have that $\bar{z} \in Z_{ad,D}$. 

It then remains to show that $(S\bar{z},\bar{z})$ fullfills the state equation
according to Definition~\ref{def:vweak_d} and $\bar{z}$ is a minimizer to 
\eqref{eq:rpDir}. In order to show that $(S\bar{z},\bar{z})$ fulfills the 
state equation, we need to focus on the identity
    \begin{equation}\label{eq:ocexist}
        \int_{Q} 
         u_n \left(-\partial_t v + (-\Delta)^s v \right)\;dxdt
         = - \int_{\Sigma} z_n \mathcal{N}_s v\;dxdt
    \end{equation}
    for all $v \in L^2((0,T);V) \cap H^1((0,T);L^2(\Om))$ with $v(T,\cdot) = 0$
    and for a.e. $t\in (0,t)$, as $n\rightarrow
    \infty$. Since $u_n := Sz_n \rightharpoonup S\bar{z} =: \bar{u}$
    in $U_D$ as $n\rightarrow\infty$ and $z_n\rightharpoonup \bar{z}$ in 
    $Z_D$ as $n\rightarrow\infty$, we can immediately take the limit 
    and conclude that $(\bar{u},\bar{z}) \in U_D\times Z_{ad,D}$ fulfills
    the state equation according to Definition~\ref{def:vweak_d}. 

    Next, that $\bar{z}$ is the minimizer of \eqref{eq:rpDir} follows from 
    the fact that $\mathcal{J}$ is weakly lower semicontinuous: 
    $\mathcal{J}$ is the sum of two weakly lower semicontinuous functions 
    (recall that the norm is continuous and convex therefore weakly lower 
    semicontinuous). 
    
    Finally, uniqueness of $\bar{z}$ follows from the stated assumptions 
    on $J$ and $\xi$ which leads to strict convexity of $\mathcal{J}$. The proof is finished.
\end{proof}

In order to derive the first order necessary optimality conditions, we 
need an expression of the adjoint operator $S^*$. We discuss this next. We notice that for every measurable set $E\subset\RR^N$, we have that
$L^2((0,T);L^2(E))=L^2((0,T)\times E)$ with equivalent norms.

\begin{lemma}\label{lem:Sstar}
    The adjoint operator $S^* : U_D \rightarrow Z_D$ for the state equation 
    \eqref{eq:Sd} is given by                 
    \[
        S^*w = -\mathcal{N}_s p \in Z_D , 
    \]
    where $w \in U_D$ and $p \in \mathbb{U}_0$ is the weak solution to the 
    problem 
    \begin{equation}\label{adj0}
      \begin{cases}
        -\partial_t p + (-\Delta)^s p &= w \quad \mbox{in } Q, \\
                    p &= 0           \quad \mbox{in } \Sigma ,  \\
                    p(T,\cdot) & = 0 \quad \mbox{in } \Om . 
     \end{cases}                
    \end{equation}     
\end{lemma}

\begin{proof}
    First of all, since $S$ is linear and bounded, it follows that  $S^*$ is well-defined.        
    Now for every $w\in U_D$ and $z \in Z_D$, we have that 
    \[
        (w,Sz)_{L^2((0,T);L^2(\Om))} = (S^*w,z)_{L^2((0,T);L^2(\Omc))} . 
    \]
    Next, testing the equation \eqref{adj0} with $Sz$ which solves the state
    equation in the very-weak sense (cf.~Definition~\ref{eq:vw_d}) we obtain 
    that 
    \begin{align*}
        (w,Sz)_{L^2((0,T);L^2(\Om))} 
        &= (-\partial_t p + (-\Delta)^s p,Sz)_{L^2((0,T);L^2(\Om))} \\
        &= -(z,\mathcal{N}_sp)_{L^2((0,T);L^2(\Omc))}
        = (z,S^*w)_{L^2((0,T);L^2(\Omc))} ,
    \end{align*}
and the proof is complete.     
\end{proof}
    
For the remainder of this section, we will assume that $\xi > 0$. 

\begin{theorem}
    Let $\mathcal{Z} \subset Z_D$ be open such that $Z_{ad,D} \subset 
    \mathcal{Z}$ and let the assumptions of Theorem~\ref{thm:docexist} 
    hold. Moreover, let $u\mapsto J(u) : U_D \rightarrow \RR$ be 
    continuously Fr\'echet differentiable with $J'(u) \in U_D$. 
    If $\bar{z}$ is a minimizer of \eqref{eq:rpDir} over $Z_{ad,D}$, then the
    first order necessary optimality conditions are given by 
        \begin{equation}\label{eq:dirfOC}
            (-\mathcal{N}_s\bar{p}+\xi\bar{z},z-\bar{z})_{L^2((0,T);L^2(\Omc))}
                \ge 0 , \quad \forall z \in Z_{ad,D}
        \end{equation}
        where $\bar{p} \in \mathbb{U}_0$ solves the adjoint equation 
        \begin{equation}\label{adj}
      \begin{cases}
        -\partial_t \bar{p} + (-\Delta)^s \bar{p} &= J'(\bar{u}) \quad \mbox{in } Q, \\
                    \bar{p} &= 0           \quad \mbox{in } \Sigma ,  \\
                    \bar{p}(T,\cdot) & = 0 \quad \mbox{in } \Om . 
     \end{cases}                
    \end{equation} 
    Finally, \eqref{eq:dirfOC} is equivalent to 
        \begin{equation}\label{eq:dproj}
            \bar{z} = \mathcal{P}_{Z_{ad}}\left(\xi^{-1} \mathcal{N}_s\bar{p} \right) ,
        \end{equation}
    where $\mathcal{P}_{Z_{ad}}$ is the projection onto the set $Z_{ad,D}$. 
    Moreover, if $J$ is convex then \eqref{eq:dirfOC} is a sufficient condition. 
\end{theorem}
\begin{proof}
    The statements are a direct consequence of the differentiability properties of 
    $J$ and the chain rule, combined with Lemma~\ref{lem:Sstar}. Indeed, let $h \in Z_{ad,D}$ be 
    given, then the directional derivative of $\mathcal{J}$ is given by 
    \begin{align*}
        \mathcal{J}'(\bar{z})h &= (J'(S\bar{z}),Sh)_{L^2((0,T);L^2(\Om))} 
                                + \xi (\bar{z},h)_{L^2((0,T);L^2(\Omc))} \\
                               &= (S^*J'(S\bar{z})+\xi\bar{z},h)_{L^2((0,T);L^2(\Om))} 
    \end{align*}
    where we have used that $J'(S\bar{z}) \in \mathcal{L}(L^2((0,T);L^2(\Om)),\RR) =
    L^2((0,T);L^2(\Om))$. Using Lemma~\ref{lem:Sstar}, the proof of the first
    part is finished. Finally, using Lemma~\ref{lem:Nmap} we have that 
    $\mathcal{N}_s\bar{p} \in L^2((0,T);L^2(\Omc))$. Then \eqref{eq:dproj} follows 
    by using \cite[Theorem~3.3.5]{HAttouch_GButtazzo_GMichaille_2014a}. The proof is finished.
\end{proof}

\subsection{Fractional Robin Optimal Control Problem}
\label{s:rob}

Next we shall focus on the Robin optimal control problem \eqref{eq:ncp}. 
We let 
    \[
        Z_R := L^2((0,T);L^2(\Omc,\mu)), \quad 
        U_R := L^2((0,T);W_{\Om,\kappa}^{s,2})\cap H^1((0,T);( W_{\Om, \kappa}^{s,2})^\star) . 
    \]        
Recall that $d\mu = \kappa dx$ with
$\kappa \in L^1(\Omc) \cap L^\infty(\Omc)$. 
Due to Theorem~\ref{Theo-sol-Ro}, the following control-to-state (solution)
map 
    \[
        S : Z_R \rightarrow U_R, \quad z \mapsto Sz =: u ,
    \]
is well-defined. In addition, $S$ is linear and continuous. Owing to the 
continuous embedding $U_R \hookrightarrow L^2((0,T);L^2(\Om))$ we can instead define 
    \[
        S : Z_R \rightarrow L^2(0,T;L^2(\Om)) .
    \]
The so-called reduced Robin exterior parabolic optimal control problem is then given by     
    \begin{equation}\label{eq:rpRob}
        \min_{z \in Z_{ad,R}} \mathcal{J}(z) := J(Sz) + \frac{\xi}{2}\|z\|^2_{Z_R} . 
    \end{equation}
The following well-posedness result holds.
\begin{theorem}\label{thm:rexist}
    Let $Z_{ad,R}$ be a convex and closed subset of $Z_R$ and let either $\xi > 0$
    or $Z_{ad,R} \subset Z_R$ be bounded. In addition, if 
    $J : L^2((0,T);L^2(\Om)) \rightarrow \RR$
    is weakly lower-semicontinuous then there exists a solution $\bar{z}$ to \eqref{eq:rpRob}
    and equivalently \eqref{eq:ncp}. If either $J$ is convex and $\xi > 0$ or $J$ is 
    strictly convex and $\xi \ge 0$ then $\bar{z}$ is unique. 
\end{theorem}
\begin{proof}
    The proof is similar to the proof of Theorem~\ref{thm:docexist}. We only discuss the part 
    where $\{z_n\}_{n\in\mathbb{N}}$ is a minimizing sequence such that 
    $z_n \rightharpoonup \bar{z}$ in $L^2((0,T);L^2(\Omc,\mu))$ as $n\rightarrow\infty$.
    Let $(Sz_n,z_n)$,  $n\in\mathbb{N}$, be the solution of \eqref{eq:Sn}. We need
    to show that this sequence converges to $(S\bar{z},\bar{z})$ in 
    $L^2((0,T);W_{\Om,\kappa}^{s,2})\cap H^1((0,T);( W_{\Om, \kappa}^{s,2})^\star)$ 
    as $n\rightarrow\infty$ and $(S\bar{z},\bar{z})$ solves \eqref{eq:Sn} in the 
    weak sense (cf.~Definition~\ref{def:weak_n}).
    Since $u_n := Sz_n \in L^2((0,T);W_{\Om,\kappa}^{s,2})\cap H^1((0,T);( W_{\Om, \kappa}^{s,2})^\star)$ 
    solves \eqref{eq:Sn}, we have that the identity
    \begin{equation}\label{eq:b}
        \langle \partial_t u_n, v\rangle + \mathcal{E}(u_n,v) = \int_{\Omc} z_n v\;d\mu ,
    \end{equation}
    holds for every $v \in W_{\Om,\kappa}^{s,2}$ and a.e. $t\in (0,T)$, 
    where $\mathcal{E}$ is as defined in 
    \eqref{form-F}. We note that the mapping $S$ is bounded due to Theorem~\ref{Theo-sol-Ro}.
    As a result, after a subsequence, if necessary, we have that $Sz_n = u_n 
    \rightharpoonup S\bar{z} = \bar{u}$ in 
    $L^2((0,T);W_{\Om,\kappa}^{s,2})\cap H^1((0,T);( W_{\Om, \kappa}^{s,2})^\star)$
    as $n\rightarrow\infty$. Then taking the limit as $n\rightarrow\infty$ in 
    \eqref{eq:b} we obtain that 
    \[
        \langle \partial_t \bar{u}, v\rangle + \mathcal{E}(\bar{u},v) = \int_{\Omc} \bar{z} v\;d\mu ,
    \]
    i.e., $(S\bar{z},\bar{z})$ solves \eqref{eq:Sn} in the weak sense 
    (cf.~Definition~\ref{def:weak_n}). The proof is finished.
\end{proof}

As in the previous section, before we state the first order optimality conditions, we
shall derive the expression of the adjoint operator $S^*$. 

\begin{lemma}
    The adjoint operator $S^* : L^2((0,T);L^2(\Om)) \rightarrow Z_R$ is given by 
    \[
        (S^*w,z)_{Z_R} = \int_\Sigma p z \; d\mu dt \quad \forall z \in Z_R , 
    \]
    where $w \in L^2((0,T);L^2(\Om))$ and 
    $p \in L^2((0,T);W_{\Om,\kappa}^{s,2})\cap H^1((0,T);( W_{\Om, \kappa}^{s,2})^\star)$
    is the weak solution to 
    \begin{equation}\label{adj0n}
      \begin{cases}
        -\partial_t p + (-\Delta)^s p &= w \quad \mbox{in } Q, \\
           \mathcal{N}_s p + \kappa p &= 0           \quad \mbox{in } \Sigma ,  \\
                    p(T,\cdot) & = 0 \quad \mbox{in } \Om . 
     \end{cases}                
    \end{equation}
\end{lemma}

\begin{proof}
    Let $w \in L^2((0,T);L^2(\Om)$ and $z \in Z_R$. Since 
    $Sz \in L^2((0,T);W_{\Om,\kappa}^{s,2})\cap H^1((0,T);( W_{\Om, \kappa}^{s,2})^\star)
    \hookrightarrow L^2((0,T);L^2(\Om))$, with the embedding being continuous, we can write 
    \[
        (w,Sz)_{L^2((0,T);L^2(\Om))} = (S^*w,z)_{Z_R} . 
    \]
    Furthermore, testing \eqref{adj0n} with $Sz = u$ we obtain that 
    \begin{align*}
        (w,Sz)_{L^2((0,T);L^2(\Om))} &= (-\partial_t p + (-\Delta)^s p , Sz)_{L^2((0,T);L^2(\Om))} \\
                                   &= \int_\Sigma z p\;d\mu dt= (S^*w,z)_{Z_R} 
    \end{align*}
    where have used the integration-by-parts in both space and time and the fact that
    $Sz = u$ solves the state equation according to Definition~\ref{def:weak_n}.
    The proof is complete. 
\end{proof}

We conclude this section with the following first order optimality conditions result
whose proof is similar to the Dirichlet case and is omitted for brevity. We shall 
assume that $\xi > 0$. 

\begin{theorem}
    Let $\mathcal{Z} \subset Z_R$ be open such that $Z_{ad,R} \subset 
    \mathcal{Z}$ and let the assumptions of Theorem~\ref{thm:rexist} 
    holds. Let $u \mapsto J(u) : L^2((0,T);L^2(\Om)) \rightarrow \RR$
    be continuously Fr\'echet differentiable with $J'(u) \in L^2((0,T);L^2(\Om))$.
    If $\bar{z}$ is a minimizer of \eqref{eq:rpRob}, then the first order necessary
    optimality conditions are given by 
    \begin{equation}\label{eq:fopcRob}
        \int_\Sigma (\bar{p}+\xi\bar{z})(z-\bar{z}) \;d\mu dt \ge 0, \quad z \in Z_{ad,R} 
    \end{equation}    
    where $\bar{p} \in L^2((0,T);W_{\Om,\kappa}^{s,2})\cap H^1((0,T);( W_{\Om, \kappa}^{s,2})^\star)$ 
    solves the adjoint equation 
    \begin{equation}\label{adj1n}
      \begin{cases}
        -\partial_t \bar{p} + (-\Delta)^s \bar{p} = J'(\bar{u}) \quad &\mbox{in } Q, \\
           \mathcal{N}_s \bar{p} + \kappa \bar{p} = 0           \quad &\mbox{in } \Sigma ,  \\
                    \bar{p}(T,\cdot)  = 0 \quad& \mbox{in } \Om . 
     \end{cases}                
    \end{equation}
    Moreover, \eqref{eq:fopcRob} is equivalent to 
    \[
        \bar{z} = \mathcal{P}_{Z_{ad,R}}(-\xi^{-1}\bar{p})
    \]
    where $\mathcal{P}_{Z_{ad},R}$ is the projection onto the set $Z_{ad,R}$. If
    $J$ is convex then \eqref{eq:fopcRob} is sufficient. 
\end{theorem}


\section{Approximation of Dirichlet Exterior Value and Control Problems}\label{s:extapp}

Recall that the Dirichlet control problem requires approximations of the nonlocal
normal derivative of the test function (cf.~\eqref{eq:vw_d}) and the nonlocal normal
derivative of the adjoint variable (cf.~\eqref{eq:dirfOC}). Nonlocal normal derivative
is a delicate object to handle both at the continuous level and at the discrete
level. Indeed, the best known regularity result for the nonlocal normal derivative 
is as given in Lemma~\ref{lem:Nmap}. Moreover, numerical approximation of this 
object is a daunting task. In order to circumvent the approximations of the nonlocal 
normal derivative both in \eqref{eq:vw_d} and \eqref{eq:dirfOC}, in this section
we propose to approximate the parabolic Dirichlet problem by the following regularized 
parabolic Dirichlet problem (or the parabolic Robin problem). Subsequently, we 
shall approximate the parabolic Dirichlet control problem by the regularized 
parabolic Dirichlet control problem. 

Let $n \in \mathbb{N}$. In this section we are interested in solutions 
$u_n$ to the regularized parabolic Dirichlet problem 

 \begin{equation}\label{eq:Sn-G-Reg}
 \begin{cases}
    \partial_t u_n + (-\Delta)^s u_n = 0 \quad &\mbox{in } Q, \\
    \mathcal{N}_s u_n  +n\kappa u_n = n\kappa z \quad &\mbox{in } \Sigma, \\
    u_n(0,\cdot) = 0 \quad &\mbox{in } \Om ,
 \end{cases}                
 \end{equation} 
that belongs to the space 
$L^2((0,T);W^{s,2}_{\Om,\kappa}\cap L^2(\Omc)) \cap H^1((0,T);
(W^{s,2}_{\Om,\kappa}\cap L^2(\Omc))^\star)$. Notice that the space 
$W^{s,2}_{\Om,\kappa}\cap L^2(\Omc)$ is endowed with the norm
\begin{align}
\|u\|_{W_{\Om,\kappa}^{s,2}\cap L^2(\Omc)}:=\left(\|u\|_{W_{\Om,\kappa}^{s,2}}^2+\|u\|_{L^2(\Omc)}^2\right)^{\frac 12},\;\;u\in W_{\Om,\kappa}^{s,2}\cap L^2(\Omc).
\end{align}
 Moreover, in our application we shall take $\kappa$ such that its support $\mbox{supp}[\kappa]$ has a positive Lebesgue measure. Thus we make the following assumption.

\begin{assumption}\label{asum}
We assume that $\kappa\in L^1(\Omc)\cap L^\infty(\Omc)$ and satisfies $\kappa>0$ almost everywhere in $K:=\mbox{supp}[\kappa]\subset\Omc$, where the Lebesgue measure $|K| > 0$. 
\end{assumption}

It follows from Assumption \ref{asum} that $\displaystyle\int_{\Omc}\kappa\;dx>0$. 

 We recall that a solution to \eqref{eq:Sn-G-Reg}
belongs to $L^2((0,T);W^{s,2}_{\Om,\kappa}) \cap H^1((0,T);
(W^{s,2}_{\Om,\kappa})^\star)$ by using Proposition~\ref{pro-sol-Ro}. In order to show that this solution
lies in $L^2((0,T);W^{s,2}_{\Om,\kappa}\cap L^2(\Omc)) \cap H^1((0,T);
(W^{s,2}_{\Om,\kappa}\cap L^2(\Omc))^\star)$ we recall a result from 
\cite[Lemma~6.2]{HAntil_RKhatri_MWarma_2018a}.
\begin{lemma}{\cite[Lemma~6.2]{HAntil_RKhatri_MWarma_2018a}}\label{le62}
Assume that Assumption \ref{asum} holds. Then
\begin{equation}\label{norm-equiv}
\|u\|_W:=\left(\int\int_{\RR^{2N}\setminus(\Omc)^2}\frac{|u(x)-u(y)|^2}{|x-y|^{N+2s}}\;dxdy+\int_{\Omc}|u|^2\;dx\right)^{\frac 12},
\end{equation}
defines an equivalent norm on $W_{\Om,\kappa}^{s,2}\cap L^2(\Omc)$.
\end{lemma}
We are now ready to state the main result of this section whose proof is
motivated by the previously considered elliptic case by the authors in
\cite{HAntil_RKhatri_MWarma_2018a}. 

 \begin{theorem}[\bf Approximation of weak solutions to Dirichlet problem]
 \label{thm:approx_dbcp}
  Let Assumption~\ref{asum} hold. Then the following assertions are true.
  \begin{enumerate}
\item  Let $z \in H^1((0,T);W^{s,2}(\Omc))$ and  $u_{n} \in L^2((0,T);W^{s,2}_{\Om,\kappa}\cap L^2(\Omc)) \cap H^1((0,T);
(W^{s,2}_{\Om,\kappa}\cap L^2(\Omc))^\star)$ be the weak solution of  \eqref{eq:Sn-G-Reg}. Let $u\in \mathbb{U}$ be the weak solution to the state equation 
  \eqref{eq:Sd}. Then there is a constant $C>0$ (independent of $n$) such that
  \begin{align}\label{es-diff}
  \|u-u_{n}\|_{L^2((0,T);L^2(\RR^N))}\le \frac{C}{n}\|u\|_{L^2((0,T);W^{s,2}(\RR^N))}.
  \end{align}
In particular $u_{n}$ converges strongly to $u$ in $L^2((0,T);L^2(\Omega))$ as $n\to\infty$.

\item Let $z \in L^2((0,T);L^2(\Omc))$ and $u_{n} \in L^2((0,T);W^{s,2}_{\Om,\kappa}\cap L^2(\Omc)) \cap H^1((0,T);
(W^{s,2}_{\Om,\kappa}\cap L^2(\Omc))^\star)$ be the weak solution of \eqref{eq:Sn-G-Reg}. Then there is a subsequence that we still denote by $\{u_n\}_{n\in\NN}$  and a $\tilde u\in L^2((0,T);L^2(\RR^N))$ such that $u_{n}\rightharpoonup \tilde u$ in $L^2((0,T);L^2(\RR^N))$ as $n\to\infty$, and $\tilde u$ satisfies

\begin{align}\label{eq63}
\int_{Q}\tilde u(-\partial_t v + (-\Delta)^sv)\;dxdt=-\int_{\Sigma}\tilde u\mathcal N_sv\;dxdt,
\end{align}
for all $v \in L^2((0,T);V) \cap H^1((0,T);L^2(\Om))$ with $v(T,\cdot) = 0$.
\end{enumerate}
 \end{theorem}
\begin{proof}
(a) 
We begin by discussing well-posedness of \eqref{eq:Sn-G-Reg}. We first
notice that under our assumption we have that $W^{s,2}(\Omc) \hookrightarrow
L^2(\Omc) \hookrightarrow L^2(\Omc,\mu)$. Now a weak solution 
$u_n \in L^2((0,T);W^{s,2}_{\Om,\kappa}\cap L^2(\Omc)) \cap H^1((0,T);
(W^{s,2}_{\Om,\kappa}\cap L^2(\Omc))^\star)$
to \eqref{eq:Sn-G-Reg} fulfills the identity 
\begin{align}\label{RP-WS}
\langle\partial_t u_n, v \rangle + 
\frac{C_{N,s}}{2}\int\int_{\RR^{2N}\setminus(\Omc)^2}&\frac{(u_{n}(x)-u_{n}(y))(v(x)-v(y))}{|x-y|^{N+2s}}\;dxdy\notag\\
&+n\int_{\Omc}u_{n} v\;d\mu=n\int_{\Omc} zv\;d\mu,
\end{align}
for every $v\in  W_{\Om,\kappa}^{s,2}\cap L^2(\Omc)$ and almost every 
$t \in (0,T)$. For every 
$n \in \mathbb{N}$, existence of a unique solution 
$u_{n} \in L^2((0,T);W^{s,2}_{\Om,\kappa}\cap L^2(\Omc)) \cap H^1((0,T);
(W^{s,2}_{\Om,\kappa}\cap L^2(\Omc))^\star)$
to \eqref{eq:Sn-G-Reg} follows by using
the arguments of Proposition~\ref{pro-sol-Ro}.  

Next we prove the estimate \eqref{es-diff}. 
For $v,w\in L^2((0,T);W^{s,2}_{\Om,\kappa}\cap L^2(\Omc)) \cap H^1((0,T);
(W^{s,2}_{\Om,\kappa}\cap L^2(\Omc))^\star)$, we shall let
\begin{align}\label{eq:ddd}
\mathcal E_n(v,w):=\frac{C_{N,s}}{2}\int\int_{\RR^{2N}\setminus(\Omc)^2}\frac{(v(x)-v(y))(w(x)-w(y))}{|x-y|^{N+2s}}\;dxdy+n\int_{\Omc}vw\;d\mu.
\end{align} 
It is not difficult to see (cf. \cite[Eq.~(6.17)]{HAntil_RKhatri_MWarma_2018a}) that 
there is a constant $C > 0$ such that 
\begin{align}\label{Ine-eq-norm}
\frac{C_{N,s}}{2}\int\int_{\RR^{2N}\setminus(\Omc)^2}\frac{|u_{n}(x)-u_{n}(y)|^2}{|x-y|^{N+2s}}\;dxdy+n\int_{\Omc}|u_{n}|^2\;dx
\le C\mathcal E_n(u_n,u_n).
\end{align}

Next let $u \in \mathbb{U}$ be the weak solution to the Dirichlet problem \eqref{eq:Sd} 
according to Definition~\ref{def:weak_d_1} and let
$v \in L^2((0,T);W^{s,2}_{\Om,\kappa}\cap L^2(\Omc)) \cap H^1((0,T);
(W^{s,2}_{\Om,\kappa}\cap L^2(\Omc))^\star)$. Using the integration by parts formula 
\eqref{Int-Part} we get that 
\begin{align}\label{IN-I}
\langle \partial_t (u-u_n), v \rangle + 
\mathcal E_n(u-u_{n},v)=&\int_{\Omega}\left( \partial_t (u-u_n) + (-\Delta)^s(u-u_{n}) \right)v\;dx+\int_{\Omc}\mathcal N_s(u-u_{n})v\;dx\notag\\
&+n\int_{\Omc}\left(u-u_{n}\right) v\;d\mu\notag\\
=&\int_{\Omega}\left( \partial_t (u-u_n) + (-\Delta)^s(u-u_{n}) \right)v\;dx +\int_{\Omc}v\mathcal N_su\;dx\notag\\
&-\int_{\Omc}\left(\mathcal N_su_{n}+n\kappa(u_{n}-z)\right)v\;dx\notag\\
=&\int_{\Omc}v\mathcal N_su\;dx.
\end{align}
Subsequently letting $v = u-u_n$ in \eqref{IN-I} and using \eqref{Ine-eq-norm} we can conclude that
there is a constant $C >0$ (independent of $n$) such that 
\begin{align*}
C \langle \partial_t (u-u_n), u-u_n \rangle + 
n \|u-u_{n}\|_{L^2(\Omc)}^2
&\le C\left( \langle \partial_t(u-u_n),u-u_n \rangle + \mathcal E_n(u-u_{n},u-u_{n}) \right) \\
&= C \int_{\Omc}(u-u_{n})\mathcal N_su\;dx\\
&\le C \|u-u_{n}\|_{L^2(\Omc)}\|\mathcal N_su\|_{L^2(\Omc)}\\
&\le C \|u-u_{n}\|_{L^2(\Omc)} \|u\|_{W^{s,2}(\RR^N)} \\
&\le \frac{n}{2} \|u-u_{n}\|_{L^2(\Omc)}^2 + \frac{C^2}{2n} \|u\|_{W^{s,2}(\RR^N)}^2.
\end{align*}
Hence,
\[
C \langle \partial_t (u-u_n), u-u_n \rangle + 
\frac{n}{2} \|u-u_{n}\|_{L^2(\Omc)}^2 \le \frac{C}{n} \|u\|_{W^{s,2}(\RR^N)}^2,
\]
where we have replaced the constant $C^2$ by $C$. Since 
\[
    \langle \partial_t (u-u_n), u-u_n \rangle = \frac{1}{2} \partial_t \|u-u_n\|_{L^2(\Omc)}^2,
\]
we arrive at 
\[
\frac{C}{2} \partial_t \|u-u_n\|_{L^2(\Omc)}^2 + 
\frac{n}{2} \|u-u_{n}\|_{L^2(\Omc)}^2 \le \frac{C}{n} \|u\|_{W^{s,2}(\RR^N)}^2 .
\]
Then
\[
\frac{C}{2} \|u-u_n\|_{L^2(\Omc)}^2 + 
\frac{n}{2} \int_0^t \|u-u_{n}\|_{L^2(\Omc)}^2 \le \frac{C}{n} \int_0^t \|u\|_{W^{s,2}(\RR^N)}^2     
\]
which implies that 
\begin{equation}\label{B1}
\begin{cases}
    \|u-u_n\|_{L^\infty((0,T);L^2(\Omc))} \le \frac{C}{\sqrt{n}}\|u\|_{L^2((0,T);W^{s,2}(\RR^N))}  \\\
 \\
    \|u-u_n\|_{L^2((0,T);L^2(\Omc))} \le \frac{C}{n}\|u\|_{L^2((0,T);W^{s,2}(\RR^N))}  . 
    \end{cases}
\end{equation}
In order  to obtain \eqref{es-diff}, it then remains to have the estimate $\|u-u_n\|_{L^2((0,T);L^2(\Om))}$. 

We notice that  $L^2((0,T);L^2(\Omega))=L^2((0,T)\times\Omega)$ with equivalent norms and

\begin{align}\label{B1-1}
\|u-u_{n}\|_{L^2((0,T);L^2(\Om))}
=\sup_{\eta\in L^2((0,T);L^2(\Om))}\frac{\left|\int_0^T\int_{\Omega} (u-u_{n})\eta\;dxdt\right|}{\|\eta\|_{L^2((0,T);L^2(\Om))}}.
\end{align}
For any $\eta \in L^2((0,T);L^2(\Om))$ let $w \in \mathbb{U}_0$ solve the dual problem
\begin{equation}\label{edp}
  \begin{cases}
    -\partial_t w + (-\Delta)^s w &= \eta \quad \mbox{in } Q, \\
                w &= 0           \quad \mbox{in } \Sigma ,  \\
                w(T,\cdot) & = 0 \quad \mbox{in } \Om . 
 \end{cases}                
 \end{equation}
It follows from Proposition~\ref{prop:weak_Dir} that there is a unique solution to 
\eqref{edp} that fulfills 
\begin{align}\label{B2}
\|w\|_{L^2((0,T);W^{s,2}(\RR^N))}\le C\|\eta\|_{L^2((0,T);L^2(\Omega))}.
\end{align}
Notice that $w \in L^2((0,T);W_0^{s,2}(\overline\Om))$ and using \eqref{IN-I} we obtain that 
\begin{align*}
&\int_0^T \int_{\Omega}(u-u_{n})(-\partial_t w + (-\Delta)^sw)\;dxdt\\
=&\int_0^T \langle \partial_t (u-u_n), w \rangle\;dt + 
\frac{C_{N,s}}{2}\int\int_{\RR^{2N}\setminus(\Omc)^2}\frac{((u-u_{n})(x)-(u-u_{n})(y))(w(x)-w(y))}{|x-y|^{N+2s}}\;dxdy\\
&-\int_0^T\int_{\Omc}(u-u_{n})\mathcal N_sw\;dxdt\\
=&\int_0^T\langle \partial_t (u-u_n), w \rangle\;dt +\int_0^T\mathcal E_n(u-u_{n},w)\;dt-\int_0^T\int_{\Omc}(u-u_{n})\mathcal N_sw\;dxdt\\
=&\int_0^T\int_{\Omc}w\mathcal N_su\;dxdt-\int_0^T\int_{\Omc}(u-u_{n})\mathcal N_sw\;dxdt\\
=&-\int_0^T\int_{\Omc}(u-u_{n})\mathcal N_sw\;dxdt.
\end{align*}
Using the preceding identity, \eqref{B1} and \eqref{B2}, we obtain that
\begin{align}\label{B3}
\left|\int_0^T \int_{\Omega}(u-u_{n})(-\partial_t w + (-\Delta)^sw)\;dxdt\right|
=&\left|\int_0^T\int_{\Omc}(u-u_{n})\mathcal N_sw\;dx\right|\notag\\
\le& \|u-u_{n}\|_{L^2((0,T);L^2(\Omc))}\|\mathcal N_sw\|_{L^2((0,T);L^2(\Omc))}\notag\\
\le  &\frac{C}{n}\|u\|_{L^2((0,T);W^{s,2}(\RR^N))}\|w\|_{L^2((0,T);W^{s,2}(\RR^N))}\notag\\
\le&  \frac{C}{n}\|u\|_{L^2((0,T);W^{s,2}(\RR^N))}\|\eta\|_{L^2((0,T);L^2(\Omega))}.
\end{align}
Using \eqref{B1-1} and \eqref{B3} we get that
\begin{align}\label{B4}
\|u-u_{n}\|_{L^2((0,T);L^2(\Omega))}\le \frac{C}{n}\|u\|_{L^2((0,T);W^{s,2}(\RR^N))}.
\end{align}
Now the estimate \eqref{es-diff} follows from \eqref{B1} and \eqref{B4} and the proof
of Part (a) is complete. \\

(b) Let $z \in L^2((0,T);L^2(\Omc))$. Using our assumption, we immediately notice that 
$L^2(\Omc)\hookrightarrow L^2(\Omc,\mu)$. In addition, $\{u_{n}\}_{n\in\NN}$ satisfies 
\eqref{RP-WS}. Then similarly to \eqref{Ine-eq-norm} we deduce that 
\begin{align*}
C \langle \partial_t u_n, u_n \rangle +n\|u_n\|_{L^2(\Omc)}^2
\le& C ( \langle \partial_t u_n, u_n \rangle + \mathcal E_n(u_n,u_n)) \\
\le& nC\|\kappa\|_{L^\infty(\Omc)}\|z\|_{L^2(\Omc)} \|u_n\|_{L^2(\Omc)},
\end{align*}
for almost every $t \in (0,T)$. 
Since $\langle \partial_t u_n, u_n \rangle = \frac{1}{2}\partial_t \|u_n\|^2_{L^2(\Omc)}$, we 
obtain that 
\begin{align}\label{weq-2es}
\|u_n\|_{L^2((0,T);L^2(\Omc))}\le C\|z\|_{L^2((0,T);L^2(\Omc))}.
\end{align}
In order to show that $\|u_n\|_{L^2((0,T);L^2(\Om))}$ is uniformly bounded, 
we can proceed as in \eqref{B4}, i.e., by using a duality argument.
Let $\eta\in L^2((0,T);L^2(\Omega))$ and $w\in \mathbb{U}_0$ be the weak solution 
of \eqref{edp}. Then using \eqref{RP-WS} and $w \in L^2((0,T);W^{s,2}_0(\overline\Om))$,
we obtain 
\begin{align*}
&\int_{Q}u_{n} \eta \;dx 
 = \int_{Q}u_{n}(-\partial_t w + (-\Delta)^sw)\;dx\\
=&\int_0^T\langle \partial_t u_n,w \rangle\;dt 
 + \frac{C_{N,s}}{2} \int_0^T\int\int_{\RR^{2N}\setminus(\Omc)^2}\frac{(u_{n}(x)-u_{n}(y))(w(x)-w(y))}{|x-y|^{N+2s}}\;dxdy\\
 &-\int_0^T\int_{\Omc}u_{n}\mathcal N_sw\;dx\\
=&-\int_0^T\int_{\Omc}u_{n}\mathcal N_sw\;dxdt.
\end{align*}
Then using the above identity, \eqref{weq-2es} and \eqref{B2} we obtain that
\begin{align}\label{wB3}
\left|\int_0^T\int_{\Omega}u_{n}\eta \;dxdt\right|
 =&\left|\int_0^T\int_{\Omc}u_{n}\mathcal N_sw\;dxdt\right|
\le \|u_{n}\|_{L^2((0,T);L^2(\Omc))}\|\mathcal N_sw\|_{L^2((0,T);L^2(\Omc))}\notag\\
\le  &C\|z\|_{L^2((0,T);L^{2}(\Omc))}\|w\|_{L^2((0,T);W^{s,2}(\RR^N))} \notag \\
\le &C\|z\|_{L^2((0,T);L^{2}(\Omc))}\|\eta\|_{L^2((0,T);L^2(\Om))} .
\end{align}
Thus
\begin{align}\label{wB4}
\|u_{n}\|_{L^2((0,T);L^2(\Omega))}\le C\|z\|_{L^2((0,T);L^{2}(\Omc))}.
\end{align}
Combing \eqref{weq-2es} and \eqref{wB4} we get that
\begin{align}\label{wB5}
\|u_{n}\|_{L^2((0,T);L^2(\RR^N))}\le C\|z\|_{L^2((0,T);L^{2}(\Omc))}.
\end{align}
Therefore the sequence $\{u_{n}\}_{n\in\NN}$ is bounded in $L^2((0,T);L^2(\RR^N))=L^2((0,T)\times\RR^N)$.
Thus, after a subsequence, if necessary, we have that $u_{n}$ converges 
weakly to some $\tilde u$ in $L^2((0,T);L^2(\RR^N))$ as $n\to\infty$. 

It then remains to show \eqref{eq63}. By \eqref{RP-WS},  for every 
$v \in L^2((0,T);V) \cap H^1((0,T);L^2(\Om))$ with $v(T,\cdot) = 0$ we have that 
\begin{align}\label{RP-WS-2}
\int_0^T \langle\partial_t u_n, v \rangle + 
\frac{C_{N,s}}{2}\int_0^T\int\int_{\RR^{2N}\setminus(\Omc)^2}&\frac{(u_{n}(x)-u_{n}(y))(v(x)-v(y))}{|x-y|^{N+2s}}\;dxdy = 0.
\end{align}
Next, applying the integration by parts formula \eqref{Int-Part} we can deduce that 
\begin{align}\label{RP-WS-3}
\int_0^T \langle\partial_t u_n, v \rangle\;dt + 
\frac{C_{N,s}}{2}\int_0^T&\int\int_{\RR^{2N}\setminus(\Omc)^2}\frac{(u_{n}(x)-u_{n}(y))(v(x)-v(y))}{|x-y|^{N+2s}}\;dxdydt \notag \\ 
&= \int_Q u_n (-\partial_t v + (-\Delta)^s v)\;dxdt
+\int_\Sigma u_{n}\mathcal N_sv\;dx,
\end{align}
for every $v \in L^2((0,T);V) \cap H^1((0,T);L^2(\Om))$ with $v(T,\cdot) = 0$. 
Combining \eqref{RP-WS-2} and \eqref{RP-WS-3} we get the identity
\begin{align}\label{WS1}
\int_Q u_n (-\partial_t v + (-\Delta)^s v)\;dxdt
+\int_\Sigma u_{n}\mathcal N_sv\;dx = 0,
\end{align}
for every $v \in L^2((0,T);V) \cap H^1((0,T);L^2(\Om))$ with $v(T,\cdot) = 0$. 
Taking the limit as $n\rightarrow \infty$ in \eqref{WS1} we obtain that 
\begin{align*}
\int_Q \tilde{u} (-\partial_t v + (-\Delta)^s v)\;dxdt
+\int_\Sigma \tilde{u}\mathcal N_sv\;dx = 0,
\end{align*}
for every $v \in L^2((0,T);V) \cap H^1((0,T);L^2(\Om))$ with $v(T,\cdot) = 0$. 
We have shown \eqref{eq63} and  the proof is finished.
\end{proof}

We conclude this section with the approximation of the parabolic Dirichlet control 
problem \eqref{eq:dcp} via the following ``regularized" (which is nothing but 
the Robin control problem) optimal control problem: Let $Z_R :=  L^2((0,T);L^2(\Omc))$ and
then 
\begin{subequations}\label{eq:ncpR}
 \begin{equation}\label{eq:JnR}
    \min_{(u,z)\in (U_R, Z_R)} J(u) + \frac{\xi}{2} \|z\|^2_{Z_R} ,
 \end{equation}
 subject to the fractional parabolic Robin exterior value problem: Find $u \in U_R$ solving 
 \begin{equation}\label{eq:SnR}
 \begin{cases}
    \partial_t u + (-\Delta)^s u = 0 \quad &\mbox{in } Q, \\
    \mathcal{N}_s u  +n\kappa u = n\kappa z \quad &\mbox{in } \Sigma, \\
    u(0,\cdot) = 0 \quad& \mbox{in } \Om ,
 \end{cases}                
 \end{equation} 
 and the control constraints 
 \begin{equation}\label{eq:ZnR}
    z \in Z_{ad,R} .
 \end{equation}
 \end{subequations}
\begin{theorem}[\bf Approximation of the parabolic Dirichlet control problem]
     The regularized control problem \eqref{eq:ncpR} admits a minimizer 
     $(z_n, u(z_n)) \in Z_{ad,R} \times L^2((0,T);W^{s,2}_{\Om,\kappa}\cap L^2(\Omc)) \cap H^1((0,T);(W^{s,2}_{\Om,\kappa}\cap L^2(\Omc))^\star)$. If 
     $Z_R = L^2((0,T);W^{s,2}(\Omc))$ and $Z_{ad,R} \subset Z_R$ is bounded
     then for any sequence $\{n_\ell\}_{\ell=1}^\infty$ with $n_\ell \rightarrow
     \infty$, there exists a subsequence still denoted by $\{n_\ell\}_{\ell=1}^\infty$
     such that $z_{n_\ell} \rightharpoonup \tilde{z}$ in $L^2((0,T);W^{s,2}(\Omc))$
     and $u(z_{n_\ell}) \rightarrow u(\tilde{z})$ in $L^2((0,T);L^2(\RR^N))$
     as $n_\ell \rightarrow \infty$ with $(\tilde{z},u(\tilde{z}))$ solving the
     parabolic Dirichlet control problem \eqref{eq:dcp} with $Z_{ad,D}$ replaced
     by $Z_{ad,R}$. 
\end{theorem}
\begin{proof}
    The proof is similar to the elliptic case \cite{HAntil_RKhatri_MWarma_2018a} with obvious modifications and has
    been omitted for brevity. 
\end{proof}

We conclude this section by writing the stationarity system corresponding to 
\eqref{eq:ncpR}: Find  
$(z,u,p) \in Z_{ad,R} \times 
\left( L^2((0,T);W^{s,2}_{\Om,\kappa}\cap L^2(\Omc)) \cap H^1((0,T);(W^{s,2}_{\Om,\kappa}\cap L^2(\Omc))^\star) \right)^2$ with $u(0,\cdot) = p(T,\cdot) = 0$ in $\Om$ such that 

 \begin{equation}\label{eq:statptR}
 \begin{cases}
\langle \partial_t u, v \rangle + 
\displaystyle \mathcal E_n(u,v)  &= \int_{\Omc} n\kappa z v\;dx, \quad \mbox{a.e. } t \in (0,T), \\
\langle -\partial_t p, w \rangle + \displaystyle\mathcal E_n(w,p)    &= \int_{\Om}J'(u)w \;dx , \quad \mbox{a.e. } t \in (0,T), \\
\displaystyle   \int_{\Sigma} (n\kappa p+\xi z) (\widetilde{z}-z)\;dx &\ge 0 , 
 \end{cases} 
 \end{equation}
for all $(\widetilde{z},v,w) \in Z_{ad} \times (W^{s,2}_{\Om,\kappa} \cap L^2(\Omc))\times  (W^{s,2}_{\Om,\kappa}\cap L^2(\Omc))$. Here $\mathcal{E}_n$ is as in \eqref{eq:ddd}.

\section{Numerical Approximations}\label{s:numerics}

In this section, we shall introduce the numerical approximation of all the problems we have 
considered so far. We remark that solving parabolic fractional PDEs is a delicate issue. One
has to assemble the integrals with singular kernels and the resulting system matrices are dense. 
On the top of that, the optimal control problem requires solving the state equation forward in
time and adjoint equation backward in time. This can be prohibitively expensive. The purpose
of this section is simply to illustrate that the numerical results are in agreement with the 
theory and to show the benefits of the fractional optimal control problem. 

The rest of the section is organized as follows: In subsection~\ref{s:num_Rob} we first 
focus on the approximations of the Robin problem which is same as the regularized Dirichlet 
problem \eqref{eq:Sn-G-Reg}. With the help of a numerical example, we illustrate the sharpness 
of Theorem~\ref{thm:approx_dbcp}. This is followed by a source identification problem in 
subsection~\ref{s:num_sour}. The numerical example presented in subsection~\ref{s:num_sour} 
clearly indicates the strength and flexibility of nonlocal problems over the local ones.

\subsection{Approximation of parabolic Dirichlet problem by parabolic Robin problem}
\label{s:num_Rob}

We begin by introducing a discrete scheme for the parabolic Robin problem \eqref{eq:Sn-G-Reg} 
and recall that we can approximate the parabolic Dirichlet problem by the parabolic Robin problem. 
Let $\widetilde\Om$ be an open bounded set that contains $\Om$, the support of $z$, and the support
of $\kappa$. We consider a conforming simplicial triangulation of $\Om$ and $\widetilde\Om\setminus\Om$ such that the
resulting partition remains admissible. Throughout we will assume that the support of $z$ and
$\kappa$ are contained in $\widetilde\Om\setminus\Om$. Let $\mathbb{V}_h$ (on $\widetilde\Om$)
be the finite element space of continuous piecewise linear functions. We use the backward-Euler 
to carry out the time discretization: Let $K$ denote the number of time intervals, we set the
time-step to be $\tau = T/K$, then for $k = 1,\dots, K$, the fully discrete approximation of 
\eqref{eq:Sn-G-Reg} with nonzero right-hand-side $f$ and initial datum $u^{(0)} = u(0,\cdot)$ 
is given by: find $u^{(k)}_h \in \mathbb{V}_h$ such that 
 \begin{align}\label{eq:discR}
 \begin{aligned}
  \int_\Om u_h^{(k)}v\;dx + 
  \tau\mathcal{E}_n(u^{(k)}_h,v) 
  &= \tau \langle f^{(k)},v\rangle 
  + \tau \int_{\widetilde\Om \setminus \Om} n\kappa z^{(k)} v\;dx \\
  &\quad+\int_\Om u_h^{(k-1)}v\;dx 
  \quad \forall v \in \mathbb{V}_h ,
 \end{aligned}   
 \end{align}    
where $\mathcal{E}_n$ is as in \eqref{eq:ddd}. 
The approximation of the double integral over $\RR^{2N}\setminus(\Omc)^2$ is carried out using
the approach of \cite{acosta2017short}. The remaining integrals are computed using 
quadrature which is accurate for polynomials of degree less than and equal to 4. All the 
implementations are carried in Matlab and we use the direct solver to solve the linear systems. 
    
We next consider an example of a parabolic Dirichlet problem with nonzero exterior conditions.
Let $\Om = B_0(1/2) \subset \RR^2$ and $T =1$, we aim to find $u$ solving 
    \begin{align}\label{eq:u1u2}
     \begin{cases}
        \partial_t u + (-\Delta)^s u &= u_{\rm exact} + e^t \quad \mbox{in } Q, \\
                        u(\cdot,t)   &= u_{\rm exact}(\cdot,t)
                                    \quad \mbox{in } \Sigma ,\\
                        u(\cdot,0)   &= u_{\rm exact}(\cdot,0)
                                    \quad \mbox{in } \Om      .       
     \end{cases}                               
    \end{align}
The exact solution for this problem is given by 
    \[
        u_{\rm exact}(x,t) 
            = \frac{2^{-2s} e^t}{\Gamma(1+s)^2} 
       \left( 1 - |x|^2 \right)_{+}^s . 
    \]

We set $\widetilde\Om = B_0(1.5)$ and approximate \eqref{eq:u1u2} 
by using \eqref{eq:discR}. Moreover, we set $\kappa = 1$. We divide
the time interval $(0,1)$ into 1800 subintervals. For a fixed $s = 0.6$ and spatial
Degrees of Freedom (DoFs) $= 6017$, we study the $L^2(0,T;L^2(\Omega))$ error 
$\|u_{\rm exact}-u_h\|_{L^2(0,T;L^2(\Omega))}$ with respect to $n$ in Figure~\ref{eq:fig_rate} (left). 
We obtain a convergence rate of $1/n$, as predicted by Theorem~\ref{thm:approx_dbcp} (a).

\begin{figure}[h!]
\centering
\includegraphics[width=0.35\textwidth]{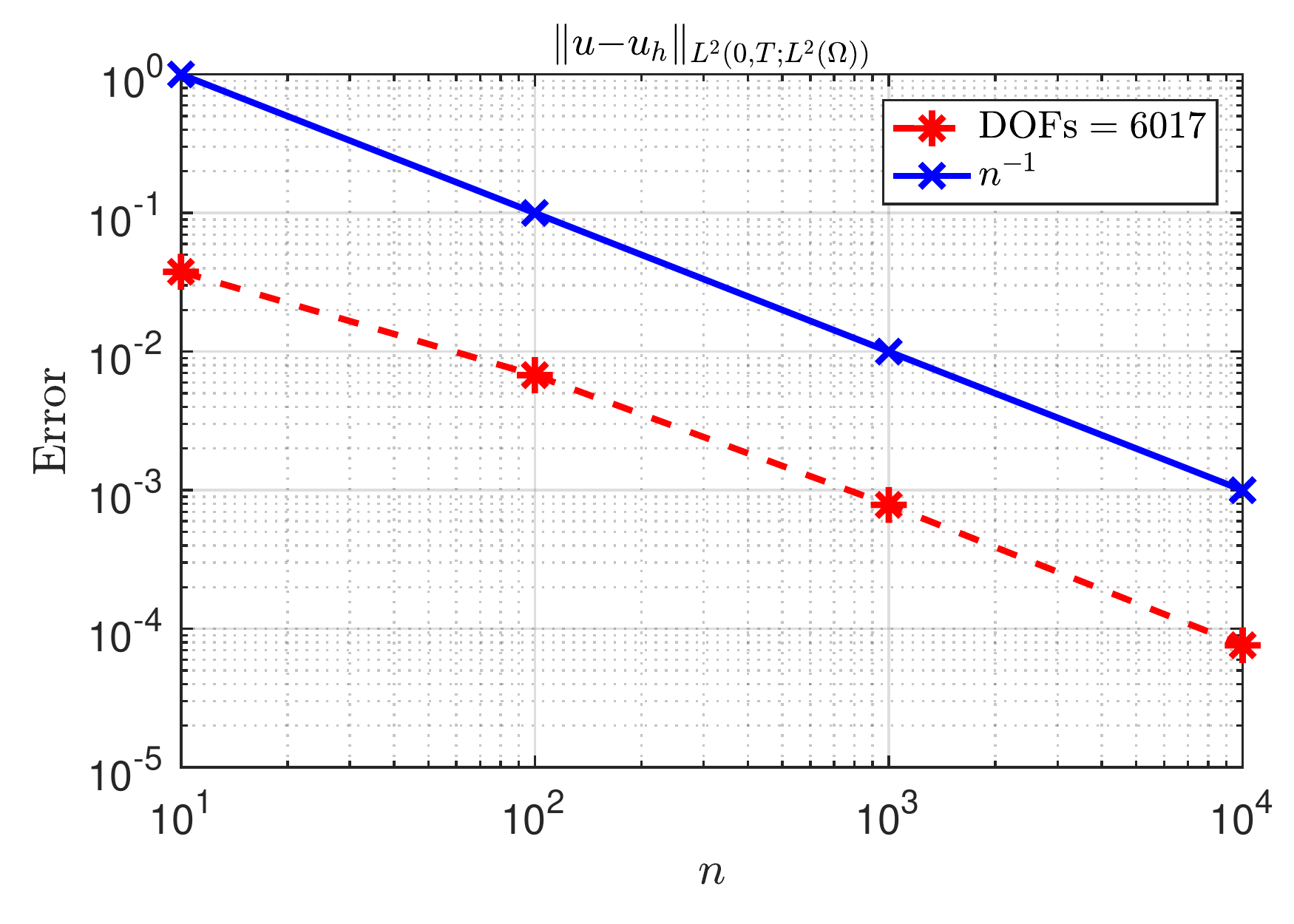}
\includegraphics[width=0.35\textwidth]{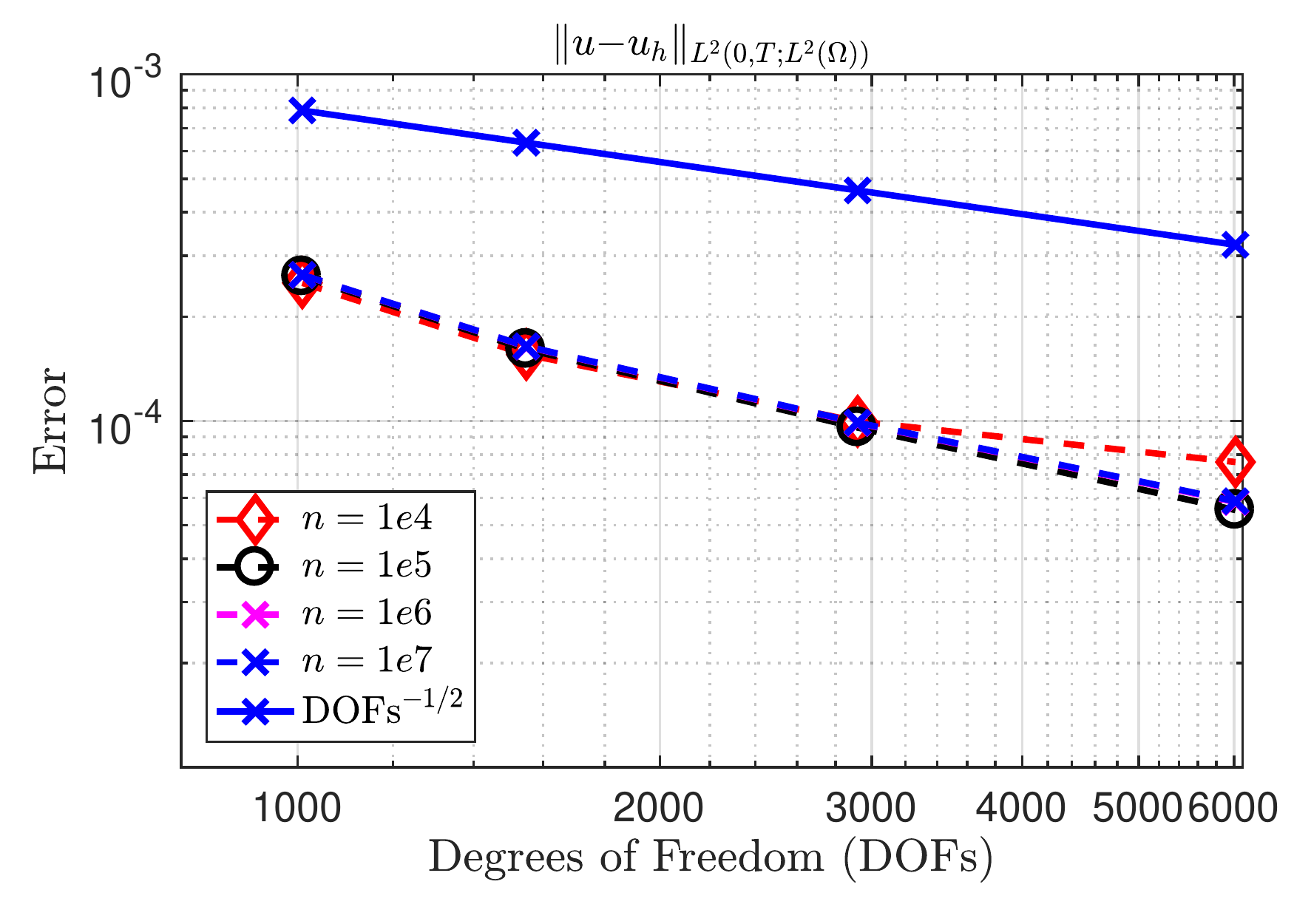}
\caption{Left panel: Fix $s = 0.6$, Degrees of Freedom (DoFs) $= 6017$.
The number of time intervals is 1800. The solid line denotes the reference line and
the dottle line is the actual error. We observe that the error 
$\|u_{\rm exact}-u_h\|_{L^2(0,T;L^2(\Omega))}$ with respect to $n$ decays at the rate of 
$1/n$ as predicted by the estimate \eqref{es-diff} in Theorem~\ref{thm:approx_dbcp}(a). 
Right panel: Let $s=0.6$ and number of time intervals = 1800, be fixed. We have shown
that the error with respect to spatial DoFs, for $n=10^4, n=10^5, n = 10^6$, and $n = 10^7$,
behaves as $({\rm DoFs})^{-\frac12}$. 
\label{eq:fig_rate}}
\end{figure}

In the right panel, in Figure~\ref{eq:fig_rate}, we have shown the error 
$\|u_{\rm exact}-u_h\|_{L^2(0,T;L^2(\Omega))}$ for a fixed $s = 0.6$, but $n = 1e4, 1e5, 1e6, 1e7$, 
as a function of DoFs. We observe that the error remains stable with
respect to $n$ as we refine the spatial mesh. Moreover, the observed rate of convergence 
is $({\rm DoFs})^{-\frac12}$.

\subsection{Parabolic source/control identification problem}
\label{s:num_sour}

After the validation in the previous example, we are now ready to consider a 
source/control identification problem where the source/control is located outside the domain
$\Om$. The optimality system is as given in \eqref{eq:statptR}. The spatial discretization 
of all the optimization variables $(u,z,p)$ is carried out using continuous
piecewise linear finite elements and time discretization using backward-Euler.
We set the objective function to be 
    \[
        j(u,z) := J(u)
               + \frac{\xi}{2} \|z\|^2_{L^2((0,T);L^2(\Omc))}, \quad 
               \mbox{with} \quad J(u) :=  \frac12 \|u-u_d\|^2_{L^2((0,T);L^2(\Om))}  ,
    \]
where $u_d : L^2((0,T);L^2(\Om)) \rightarrow \mathbb{R}$ is the given data (observations). 
Moreover, we let $Z_{ad,R} := \{ z\in L^2((0,T);L^2(\Omc)) \;:\; z \ge 0, \ \mbox{a.e. in } (0,T) \times \widehat{\Om} \}$  
where $\widehat{\Om}$ is the support set of the control $z$ that is contained in 
$\widetilde\Om\setminus \Om$. We solve the optimization problem using projected-BFGS method 
with Armijo line search. 

We consider the domain as given in Figure~\ref{f:ex2_setup}. The circle denotes 
$\widetilde\Om = B_0(3/2)$ and the larger square denotes the domain $\Om = [-0.4,0.4]^2$. 
The smaller square, inside $\widetilde\Om \setminus \Om$, is $\widehat\Om$ is where
the source/control is supported. The right panel shows a finite element mesh with 
DoFs = 6103. 
 \begin{figure}[htb] 
  \centering
  \includegraphics[width = 0.22\textwidth]{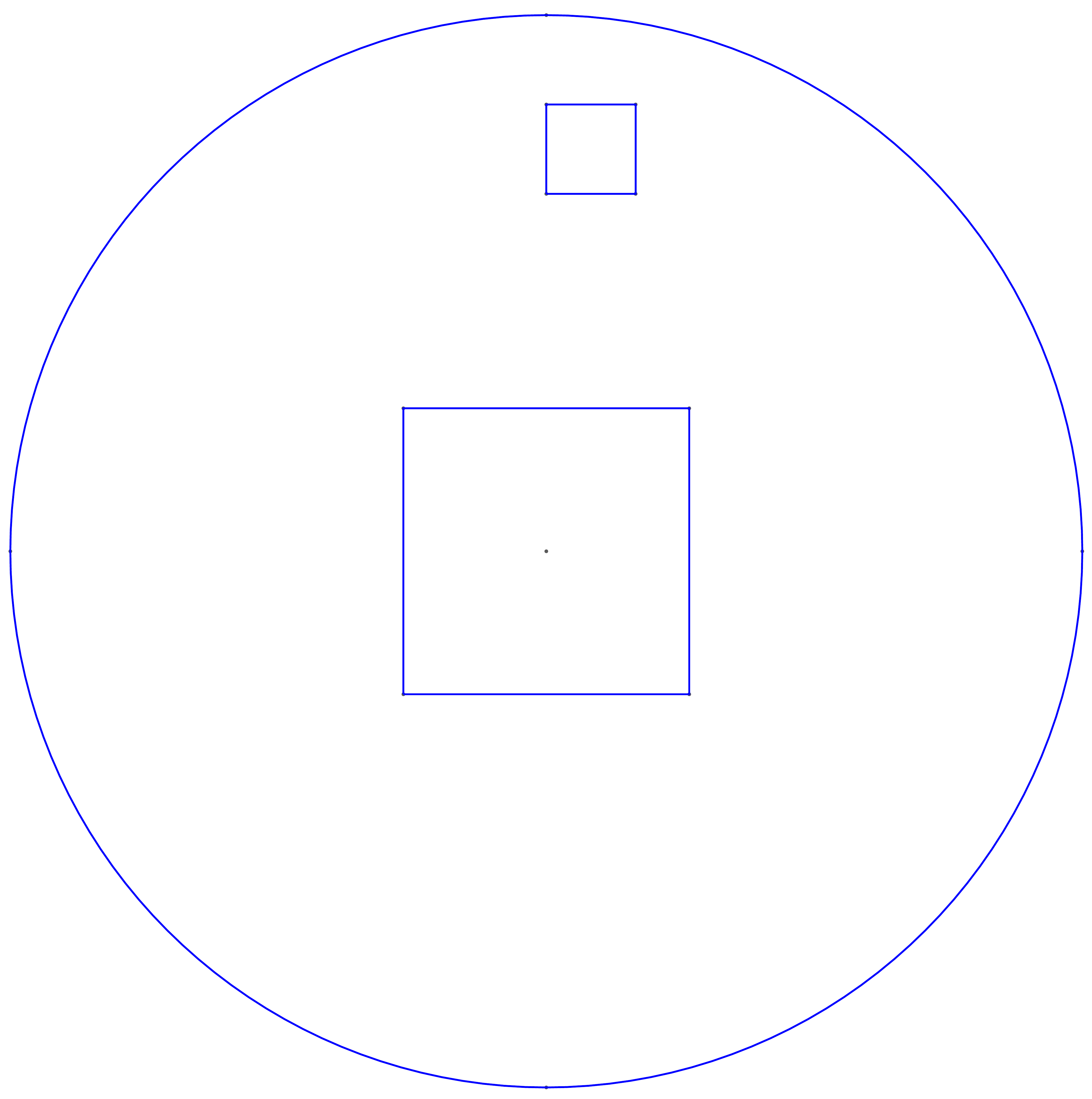}\qquad 
  \includegraphics[width = 0.25\textwidth]{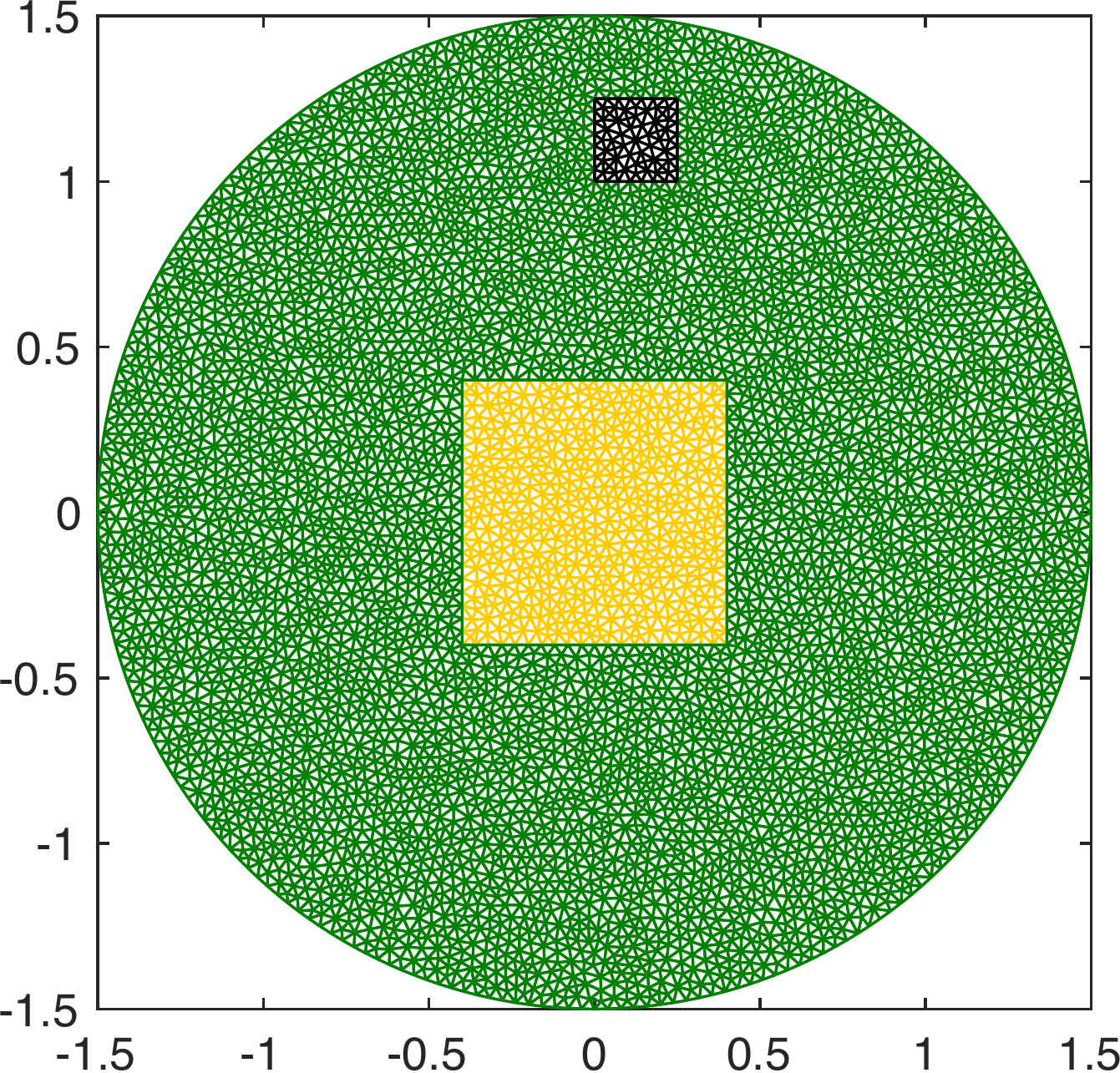}
  \caption{\label{f:ex2_setup}
  {\bf Left panel:} The circle denotes $\widetilde\Om$ and the larger 
  square denotes the domain $\Omega$. Moreover, the outer square inside 
  $\widetilde\Om \setminus \Om$ is $\widehat\Om$, i.e., the region 
  where the source/control is supported. {\bf Right panel}: A finite element mesh.}
 \end{figure}

We generate the data $u_d$ as follows: for $z = 1$, we solve the state equation
(first equation in \eqref{eq:statptR}). We then add a normally distributed noise
with mean zero and standard deviation 0.005. We call the resulting expression 
$u_d$. In addition, we set $\kappa = 1$ and $n = 1e7$. 

Next, we identify the source $\bar{z}_h$ by solving the optimality system  
\eqref{eq:statptR}. For $\xi = 1e$-8, our results are shown in 
Figure~\ref{fig:identification}. In the first two rows, we have plotted 
$\bar{z}_h$ for $s = 0.1$ at 4 time instances $t = 0.25, 0.3, 0.43, 0.58$. 
The third row shows $\bar{z}_h$ for $s = 0.8$ at only one of these time
instances since $\bar{z}_h$ is zero at the remaining three time instances.
This is not surprising, since as $s$ approaches $1$, the fractional Laplacian
approaches the standard Laplacian which does not allow a source placement outside
the closure of $\Omega$. 

\begin{figure}[h!]
\centering
\includegraphics[width=0.4\textwidth]{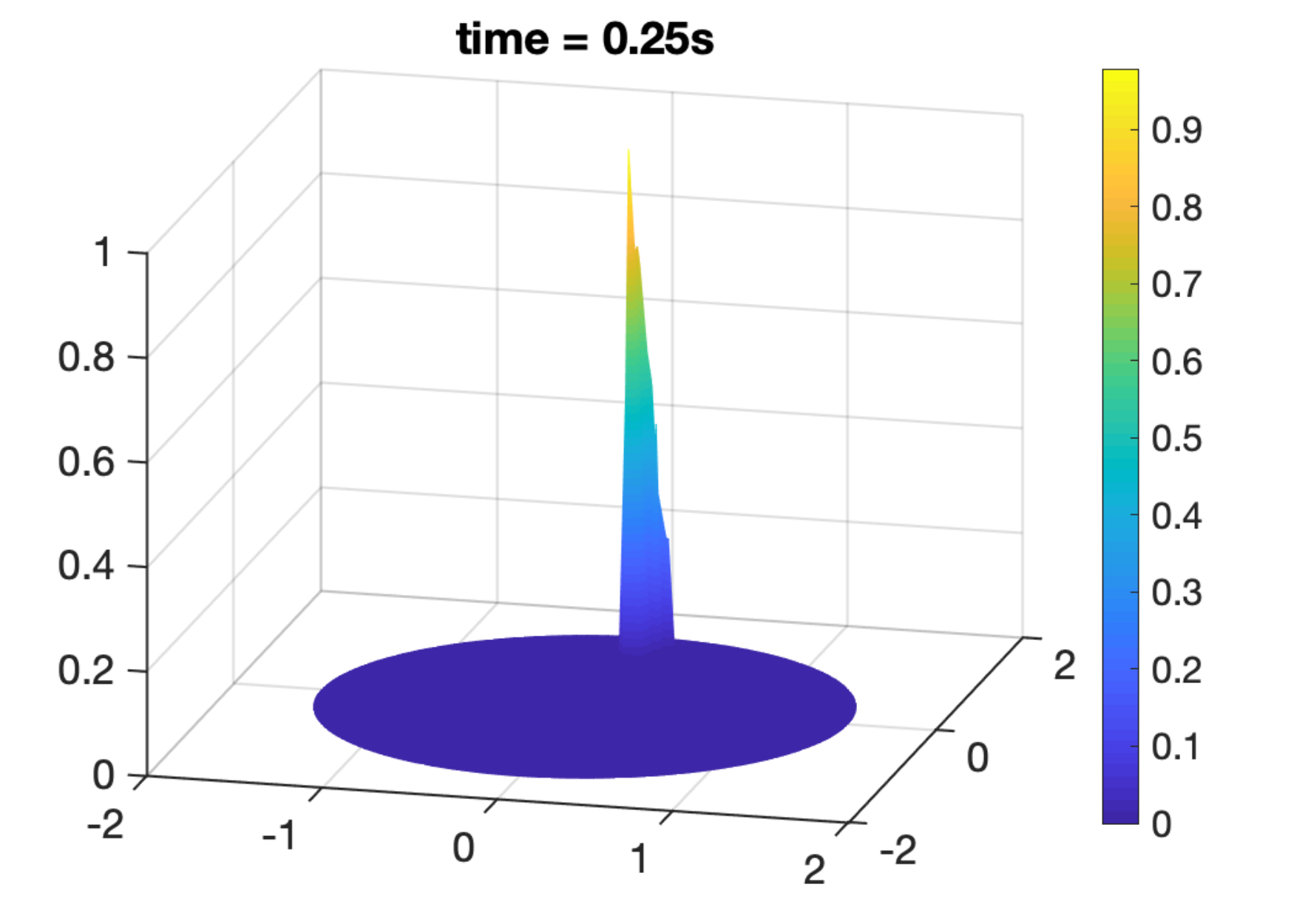}
\includegraphics[width=0.4\textwidth]{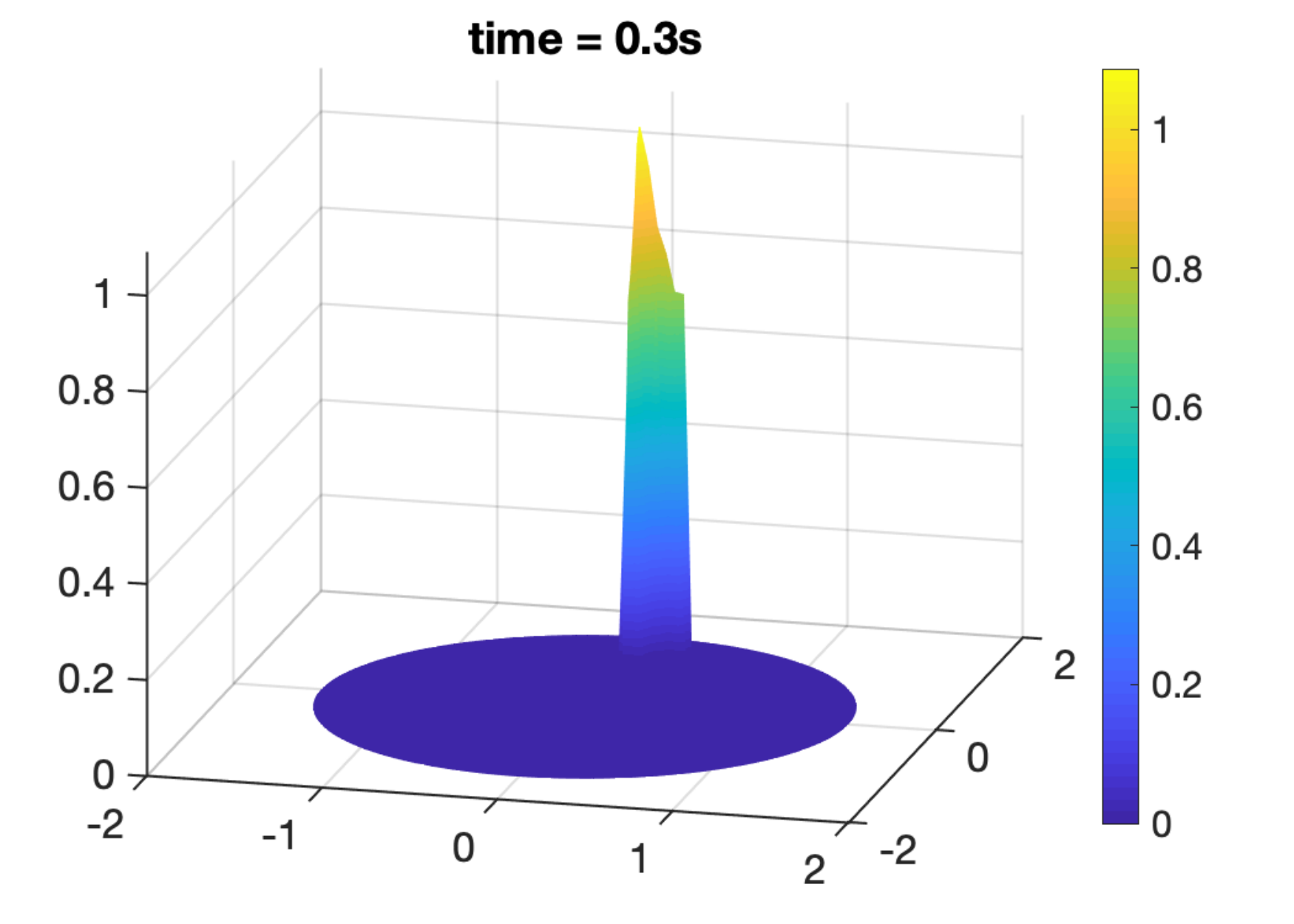}
\includegraphics[width=0.4\textwidth]{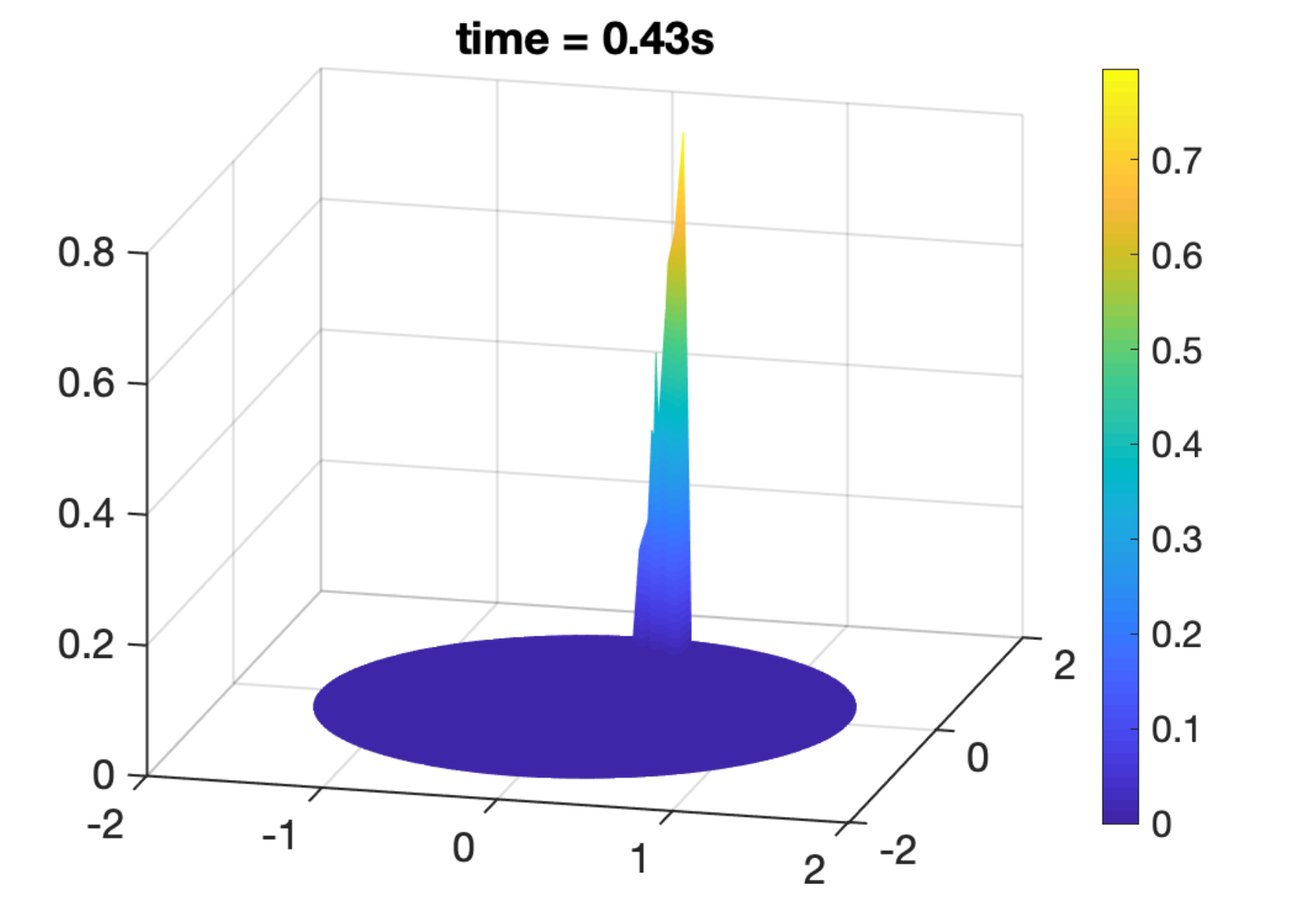}
\includegraphics[width=0.4\textwidth]{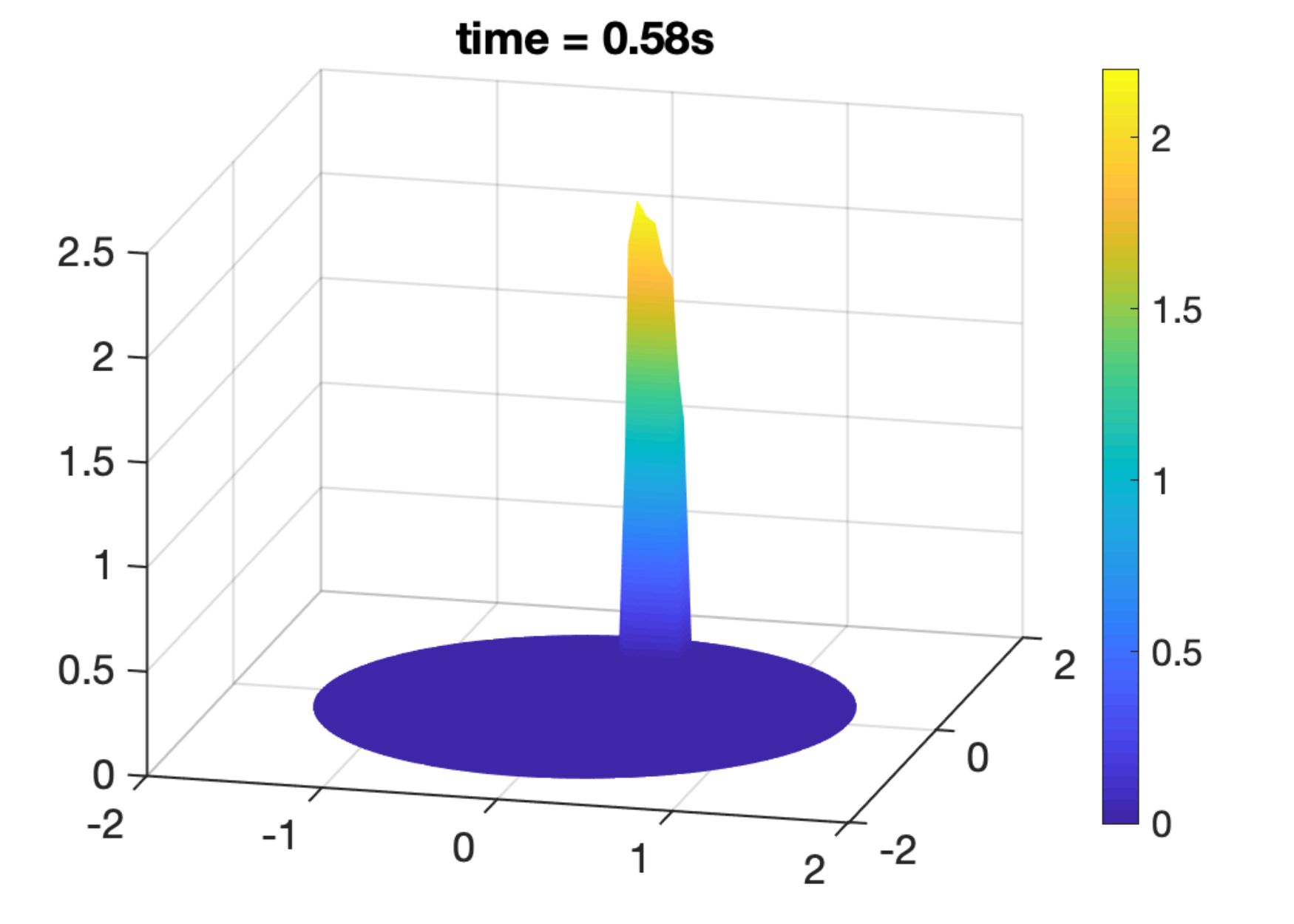}
\includegraphics[width=0.4\textwidth]{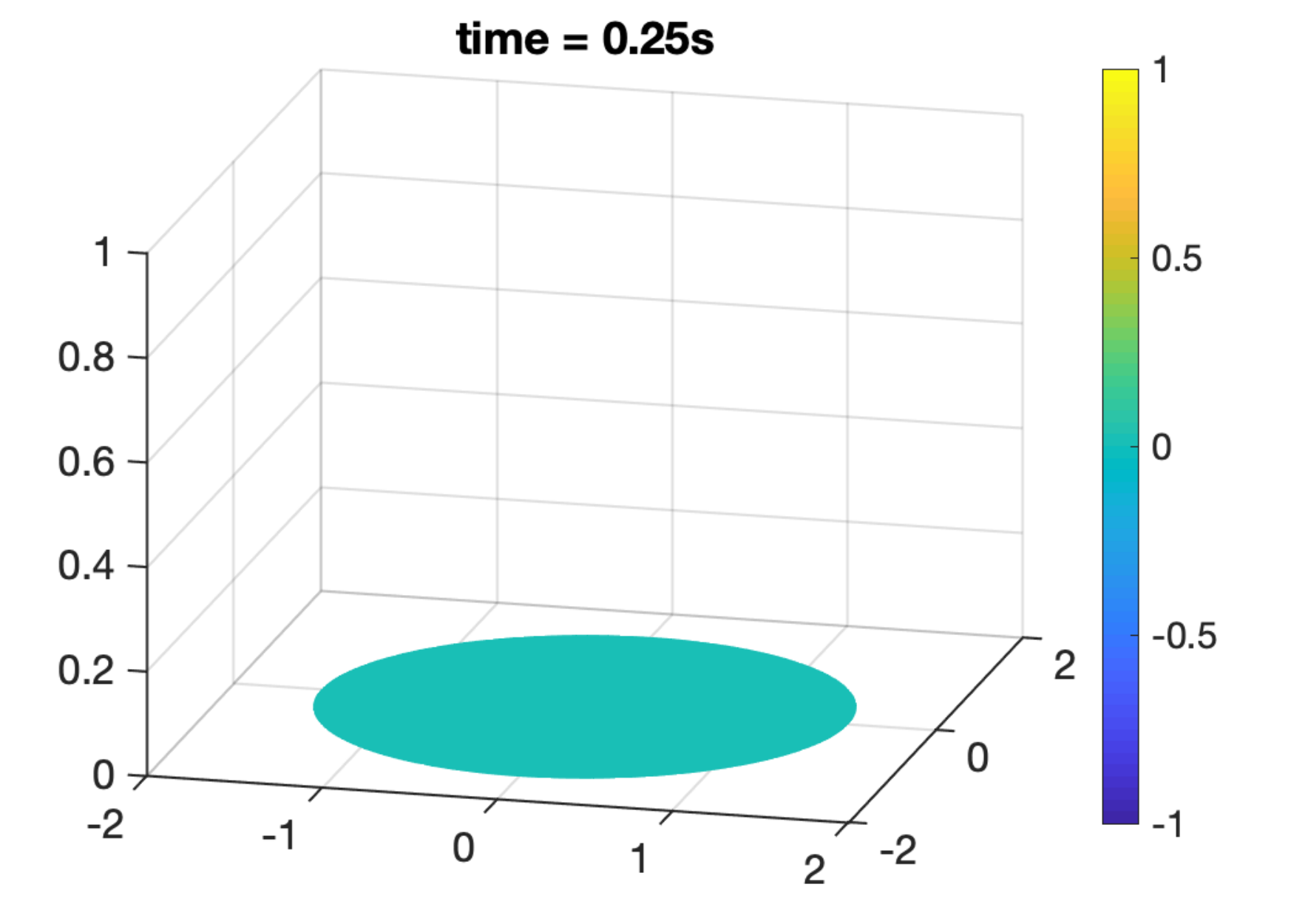}
\caption{\label{fig:identification}The first and second row shows the source 
$\bar{z}_h$ 
for exponent $s=0.1$ at 4 different time instances, $t = 0.25s, 0.3s, 0.43s, t=0.58s$.
The last row shows $\bar{z}_h$ for exponent $s = 0.8$ at $t = 0.25s$. Notice that we 
have only shown one frame for $s = 0.8$ because $\bar{z}_h \equiv 0$ 
at $t = 0.3s, 0.43s, t=0.58s$. This comparison between $\bar{z}_h$ for 
$s=0.1$ and $s=0.8$ clearly indicates that we can recover the sources for 
smaller values of $s$ but $s$ approaches 1, since the fractional Laplacian
approaches the standard Laplacian, we cannot see the external source at 
all times.}
\end{figure}

\bibliographystyle{plain}
\bibliography{refs}

\def\ocirc#1{\ifmmode\setbox0=\hbox{$#1$}\dimen0=\ht0 \advance\dimen0
  by1pt\rlap{\hbox to\wd0{\hss\raise\dimen0
  \hbox{\hskip.2em$\scriptscriptstyle\circ$}\hss}}#1\else {\accent"17 #1}\fi}
  \def\cprime{$'$} \def\ocirc#1{\ifmmode\setbox0=\hbox{$#1$}\dimen0=\ht0
  \advance\dimen0 by1pt\rlap{\hbox to\wd0{\hss\raise\dimen0
  \hbox{\hskip.2em$\scriptscriptstyle\circ$}\hss}}#1\else {\accent"17 #1}\fi}
\begin{thebibliography}{10}

\bibitem{acosta2017short}
G.~Acosta, F.M. Bersetche, and J.P. Borthagaray.
\newblock A short fe implementation for a 2d homogeneous dirichlet problem of a
  fractional laplacian.
\newblock {\em Computers \& Mathematics with Applications}, 74(4):784--816,
  2017.

\bibitem{antil2017spectral}
H.~Antil and S.~Bartels.
\newblock Spectral {A}pproximation of {F}ractional {PDE}s in {I}mage
  {P}rocessing and {P}hase {F}ield {M}odeling.
\newblock {\em Comput. Methods Appl. Math.}, 17(4):661--678, 2017.

\bibitem{HAntil_SBartels_GDogan_2019a}
H.~Antil, S.~Bartels, and G.~Dogan.
\newblock A phase field segmentation model with fractional diffusion for
  improved boundary regularization.
\newblock {\em Submitted}, 2019.

\bibitem{antil2018fractional}
H.~Antil, T.~Berry, and J.~Harlim.
\newblock Fractional diffusion maps.
\newblock {\em arXiv preprint arXiv:1810.03952}, 2018.

\bibitem{HAntil_RKhatri_MWarma_2018a}
H.~Antil, R.~Khatri, and M.~Warma.
\newblock External optimal control of nonlocal pdes.
\newblock {\em Inverse Problems, to appear}, 2019.

\bibitem{HAntil_RHNochetto_PVenegas_2018a}
H.~Antil, R.H. Nochetto, and P.~Venegas.
\newblock Controlling the {K}elvin force: basic strategies and applications to
  magnetic drug targeting.
\newblock {\em Optim. Eng.}, 19(3):559--589, 2018.

\bibitem{HAntil_RHNochetto_PVenegas_2018b}
H.~Antil, R.H. Nochetto, and P.~Venegas.
\newblock Optimizing the {K}elvin force in a moving target subdomain.
\newblock {\em Math. Models Methods Appl. Sci.}, 28(1):95--130, 2018.

\bibitem{antil2017fractional}
H.~Antil, J.~Pfefferer, and S.~Rogovs.
\newblock Fractional operators with inhomogeneous boundary conditions:
  Analysis, control, and discretization.
\newblock {\em Communications in Mathematical Sciences (CMS)},
  16(5):1395--1426, 2018.

\bibitem{AnPfWa2017}
H.~Antil, J.~Pfefferer, and M.~Warma.
\newblock A note on semilinear fractional elliptic equation: analysis and
  discretization.
\newblock {\em ESAIM: M2AN}, 51(6):2049--2067, 2017.

\bibitem{HAntil_CNRautenberg_2018b}
H.~Antil and C.N. Rautenberg.
\newblock Sobolev spaces with non-{M}uckenhoupt weights, fractional elliptic
  operators, and applications.
\newblock {\em SIAM Journal on Mathematical Analysis (SIMA), to appear}, 2019.

\bibitem{antil2018b}
H.~Antil and M.~Warma.
\newblock Optimal control of the coefficient for fractional $\{$$ p
  $$\}$-$\{$L$\}$ aplace equation: Approximation and convergence.
\newblock {\em RIMS K\^{o}ky\^{u}roku}, 2090:102--116, 2018.

\bibitem{antil2017optimal}
H.~Antil and M.~Warma.
\newblock Optimal control of fractional semilinear {PDE}s.
\newblock {\em To appear: Control, Optimisation and Calculus of Variations
  (ESAIM: COCV)}, 2019.

\bibitem{antil2016optimal}
H.~Antil and M.~Warma.
\newblock Optimal control of the coefficient for regional fractional
  {p}-{L}aplace equations: Approximation and convergence.
\newblock {\em Math. Control Relat. Fields.}, 9(1):1--38, 2019.

\bibitem{ABHN}
W.~Arendt, C.~J.~K. Batty, M.~Hieber, and F.~Neubrander.
\newblock {\em Vector-valued {L}aplace transforms and {C}auchy problems},
  volume~96 of {\em Monographs in Mathematics}.
\newblock Birkh\"{a}user/Springer Basel AG, Basel, second edition, 2011.

\bibitem{ArNi}
W.~Arendt and R.~Nittka.
\newblock Equivalent complete norms and positivity.
\newblock {\em Arch. Math. (Basel)}, 92(5):414--427, 2009.

\bibitem{HAttouch_GButtazzo_GMichaille_2014a}
H.~Attouch, G.~Buttazzo, and G.~Michaille.
\newblock {\em Variational analysis in {S}obolev and {BV} spaces}, volume~17 of
  {\em MOS-SIAM Series on Optimization}.
\newblock Society for Industrial and Applied Mathematics (SIAM), Philadelphia,
  PA; Mathematical Optimization Society, Philadelphia, PA, second edition,
  2014.
\newblock Applications to PDEs and optimization.

\bibitem{BWZ}
U.~Biccari, M.~Warma, and E.~Zuazua.
\newblock Local regularity for fractional heat equations.
\newblock In {\em Recent Advances in PDEs: Analysis, Numerics and Control},
  pages 233--249. Springer, 2018.

\bibitem{BCF}
C.~Bjorland, L.~Caffarelli, and A.~Figalli.
\newblock Nonlocal tug-of-war and the infinity fractional {L}aplacian.
\newblock {\em Comm. Pure Appl. Math.}, 65(3):337--380, 2012.

\bibitem{BPS}
L.~Brasco, E.~Parini, and M.~Squassina.
\newblock Stability of variational eigenvalues for the fractional
  {$p$}-{L}aplacian.
\newblock {\em Discrete Contin. Dyn. Syst.}, 36(4):1813--1845, 2016.

\bibitem{Caf3}
L.~Caffarelli and L.~Silvestre.
\newblock An extension problem related to the fractional {L}aplacian.
\newblock {\em Comm. Partial Differential Equations}, 32(7-9):1245--1260, 2007.

\bibitem{Caf1}
L.A. Caffarelli, J.-M. Roquejoffre, and Y.~Sire.
\newblock Variational problems for free boundaries for the fractional
  {L}aplacian.
\newblock {\em J. Eur. Math. Soc. (JEMS)}, 12(5):1151--1179, 2010.

\bibitem{Caf2}
L.A. Caffarelli, S.~Salsa, and L.~Silvestre.
\newblock Regularity estimates for the solution and the free boundary of the
  obstacle problem for the fractional {L}aplacian.
\newblock {\em Invent. Math.}, 171(2):425--461, 2008.

\bibitem{CDV18}
A.~Carbotti, S.~Dipierro, and E.~Valdinoci.
\newblock Local density of solutions of time and space fractional equations.
\newblock {\em arXiv preprint arXiv:1810.08448}, 2018.

\bibitem{NPV}
E.~Di~Nezza, G.~Palatucci, and E.~Valdinoci.
\newblock Hitchhiker's guide to the fractional {S}obolev spaces.
\newblock {\em Bull. Sci. Math.}, 136(5):521--573, 2012.

\bibitem{SDipierro_XRosOton_EValdinoci_2017a}
S.~Dipierro, X.~Ros-Oton, and E.~Valdinoci.
\newblock Nonlocal problems with {N}eumann boundary conditions.
\newblock {\em Rev. Mat. Iberoam.}, 33(2):377--416, 2017.

\bibitem{DSV16}
S.~Dipierro, O.~Savin, and E.~Valdinoci.
\newblock Local approximation of arbitrary functions by solutions of nonlocal
  equations.
\newblock {\em The Journal of Geometric Analysis}, pages 1--28, 2016.

\bibitem{QDu_MGunzburger_RBLehoucq_KZhou_2013a}
Q.~Du, M.~Gunzburger, R.B. Lehoucq, and K.~Zhou.
\newblock A nonlocal vector calculus, nonlocal volume-constrained problems, and
  nonlocal balance laws.
\newblock {\em Math. Models Methods Appl. Sci.}, 23(3):493--540, 2013.

\bibitem{GLX17}
T.~Ghosh, Y-H. Lin, and J.~Xiao.
\newblock The {C}alder\'{o}n problem for variable coefficients nonlocal
  elliptic operators.
\newblock {\em Comm. Partial Differential Equations}, 42(12):1923--1961, 2017.

\bibitem{ghosh2016calder}
T.~Ghosh, M.~Salo, and G.~Uhlmann.
\newblock The calder$\backslash$'on problem for the fractional
  schr$\backslash$" odinger equation.
\newblock {\em arXiv preprint arXiv:1609.09248}, 2016.

\bibitem{WGong_MHinze_ZZhou_2016a}
W.~Gong, M.~Hinze, and Z.~Zhou.
\newblock Finite element method and a priori error estimates for {D}irichlet
  boundary control problems governed by parabolic {PDE}s.
\newblock {\em J. Sci. Comput.}, 66(3):941--967, 2016.

\bibitem{Grub}
G.~Grubb.
\newblock Fractional {L}aplacians on domains, a development of {H}\"ormander's
  theory of {$\mu$}-transmission pseudodifferential operators.
\newblock {\em Adv. Math.}, 268:478--528, 2015.

\bibitem{GGrubb_2015a}
G.~Grubb.
\newblock Fractional {L}aplacians on domains, a development of {H}\"ormander's
  theory of {$\mu$}-transmission pseudodifferential operators.
\newblock {\em Adv. Math.}, 268:478--528, 2015.

\bibitem{Kry18}
N.V. Krylov.
\newblock On the paper: "all functions are locally s-harmonic up to a small
  error" by {D}ipierro, {S}avin, and {V}aldinoci.
\newblock {\em arXiv preprint arXiv:1810.07648}, 2018.

\bibitem{LL17}
Ru-Yu Lai and Yi-Hsuan Lin.
\newblock Global uniqueness for the fractional semilinear {S}chr\"{o}dinger
  equation.
\newblock {\em Proc. Amer. Math. Soc.}, 147(3):1189--1199, 2019.

\bibitem{larkin1999direct}
P.A. Larkin and M.~Whalen.
\newblock Direct, near field acoustic testing.
\newblock Technical report, SAE technical paper, 1999.

\bibitem{TLeonori_IPeral_APrimo_FSoria_2015a}
T.~Leonori, I.~Peral, A.~Primo, and F.~Soria.
\newblock Basic estimates for solutions of a class of nonlocal elliptic and
  parabolic equations.
\newblock {\em Discrete Contin. Dyn. Syst.}, 35(12):6031--6068, 2015.

\bibitem{CLR-MW}
C.~Louis-Rose and M.~Warma.
\newblock Approximate controllability from the exterior of space-time
  fractional wave equations.
\newblock {\em Applied Mathematics \& Optimization}, pages 1--44, 2018.

\bibitem{lubbe1996clinical}
A.S. L{\"u}bbe, C.~Bergemann, H.~Riess, F.~Schriever, P.~Reichardt,
  K.~Possinger, M.~Matthias, B.~D{\"o}rken, F.~Herrmann, R.~G{\"u}rtler, et~al.
\newblock Clinical experiences with magnetic drug targeting: a phase i study
  with 4'-epidoxorubicin in 14 patients with advanced solid tumors.
\newblock {\em Cancer research}, 56(20):4686--4693, 1996.

\bibitem{Nittka}
R.~Nittka.
\newblock Inhomogeneous parabolic {N}eumann problems.
\newblock {\em Czechoslovak Math. J.}, 64(139)(3):703--742, 2014.

\bibitem{RS-DP}
X.~Ros-Oton and J.~Serra.
\newblock The extremal solution for the fractional {L}aplacian.
\newblock {\em Calc. Var. Partial Differential Equations}, 50(3-4):723--750,
  2014.

\bibitem{RS17}
A.~R{\"u}land and M.~Salo.
\newblock The fractional {C}alder\'on problem: low regularity and stability.
\newblock {\em arXiv preprint arXiv:1708.06294}, 2017.

\bibitem{SV2}
R.~Servadei and E.~Valdinoci.
\newblock On the spectrum of two different fractional operators.
\newblock {\em Proc. Roy. Soc. Edinburgh Sect. A}, 144(4):831--855, 2014.

\bibitem{MIVisik_GIEskin_1965a}
M.I. Vi\v{s}ik and G.I. \`Eskin.
\newblock Convolution equations in a bounded region.
\newblock {\em Uspehi Mat. Nauk}, 20(3 (123)):89--152, 1965.

\bibitem{War-DN1}
M.~Warma.
\newblock A fractional {D}irichlet-to-{N}eumann operator on bounded {L}ipschitz
  domains.
\newblock {\em Commun. Pure Appl. Anal.}, 14(5):2043--2067, 2015.

\bibitem{War}
M.~Warma.
\newblock The fractional relative capacity and the fractional {L}aplacian with
  {N}eumann and {R}obin boundary conditions on open sets.
\newblock {\em Potential Anal.}, 42(2):499--547, 2015.

\bibitem{warma2018approximate}
M.~Warma.
\newblock Approximate controllabilty from the exterior of space-time fractional
  diffusive equations.
\newblock {\em SIAM Journal on Control and Optimization (SICON), to appear},
  2019.

\bibitem{CWeiss_BvBWaanders_HAntil_2018a}
C.~Weiss, B.~van~Bloemen Waanders, and H.~Antil.
\newblock Fractional operators applied to geophysical electromagnetics.
\newblock {\em Submitted}, 2019.

\end{thebibliography}

\end{document}